\documentclass[12pt,a4paper]{article}
\usepackage{natbib}

\usepackage{easyReview}
\usepackage{xcolor}
\usepackage[margin=1in]{geometry} 
\usepackage{amsfonts,amscd,amssymb, mathtools,mathrsfs, dsfont, bbm,bbding} 
\usepackage[amsmath,amsthm,thmmarks]{ntheorem} 
\usepackage{graphicx,xypic,color,float} 
\usepackage{indentfirst}
\usepackage{lmodern}
\usepackage{enumerate,listings,verbatim,paralist}
\usepackage{setspace,xspace}
\usepackage[colorlinks=true,citecolor=blue]{hyperref}
\usepackage{extarrows}
\usepackage{pdfpages}
\usepackage[normalem]{ulem}
\usepackage{subfigure}
\usepackage{bm}
\usepackage[ruled]{algorithm2e}
\usepackage{diagbox}

\usepackage{setspace}
\AtBeginDocument{%
	\addtolength\abovedisplayskip{-0.15\baselineskip}%
	\addtolength\belowdisplayskip{-0.15\baselineskip}%
	\addtolength\abovedisplayshortskip{-0.15\baselineskip}%
	\addtolength\belowdisplayshortskip{-0.15\baselineskip}%
}

\usepackage[toc,page]{appendix}

\newtheorem{theorem}{Theorem} [section]

\newtheorem{lemma}[theorem]{Lemma}

\newtheorem{remark}{Remark}[section]

\numberwithin{equation}{section} 

\renewcommand{\geq}{\geqslant}
\renewcommand{\leq}{\leqslant}

\newcommand{\citethm}[1]{Theorem \ref{#1}}

\newcommand{\opfont}{\mathbb}

\newcommand{\BE}[2][]{\ensuremath{\operatorname{\opfont{E}}^{#1}\!\left[#2\right]}}
\newcommand{\bp}{\ensuremath{\opfont{P}}}
\newcommand{\BP}[2][]{\ensuremath{\operatorname{\opfont{P}}^{#1}\!\left(#2\right)}}

\newcommand{\BH}{\ensuremath{\mathscr{H}}}

\newcommand{\R}{\ensuremath{\operatorname{\mathbb{R}}}}

\newcommand{\argmax}{\ensuremath{\operatorname*{argmax}}}

\newcommand{\dd}{\ensuremath{\operatorname{d}\! }}
\newcommand{\dt}{\ensuremath{\operatorname{d}\! t}}
\newcommand{\ds}{\ensuremath{\operatorname{d}\! s}}
\newcommand{\dr}{\ensuremath{\operatorname{d}\! r}}

\newcommand{\dv}{\ensuremath{\operatorname{d}\! v}}

\newcommand{\dz}{\ensuremath{\operatorname{d}\! z}}

\newcommand{\dw}{\ensuremath{\operatorname{d}\! W}}

\newcommand{\dS}{\ensuremath{\operatorname{d}\! S}}

\newcommand{\setu}{\mathscr{U}}

\newcommand{\barv}{\bar{v}}
\newcommand{\barJ}{\bar{J}}

\newcommand{\blue}[1]{{\color{black}#1}}

\newcommand{\revise}[1]{{\color{black}#1}}

\newcommand{\nn}{\nonumber}

\newcommand{\ep}{\varepsilon}
\newcommand{\opla}{\mathscr{L}}

\newcommand{\lam}{\lambda}

\newcommand{\footremember}[2]{
	\footnote{#2}
	\newcounter{#1}
	\setcounter{#1}{\value{footnote}}
}

\usepackage{longtable}

\usepackage{booktabs}

\begin{document}
\title{Learning to Optimally Stop Diffusion Processes, with Financial Applications\footnote{We thank the participants in the First INFORMS Conference on Financial Engineering and FinTech for their comments.
}
}
\author{
	Min Dai\footremember{alley}{Department of Applied Mathematics and School of Accounting and Finance, The Hong Kong Polytechnic University, Hung Hom, Kowloon, Hong Kong, China. Email: mindai@polyu.edu.hk}
	\and Yu Sun\footremember{trailer}{Department of Applied Mathematics, The Hong Kong Polytechnic University, Hung Hom, Kowloon, Hong Kong, China. Email: yusun017@gmail.com}
	\and Zuo Quan Xu\footremember{Germany}{Department of Applied Mathematics, The Hong Kong Polytechnic University, Hung Hom, Kowloon, Hong Kong, China. Email: maxu@polyu.edu.hk}
	\and Xun Yu Zhou\footremember{Mexico}{Department of Industrial Engineering and Operations Research \& Data Science Institute, Columbia University, New York, USA, NY 10027. Email: xz2574@columbia.edu}
}
\date{\revise{August 11, 2025}}
\maketitle

\begin{abstract}
	We study optimal stopping for diffusion processes with unknown model primitives within the continuous-time reinforcement learning (RL) framework developed by \cite{wang2020reinforcement}, and present applications to option pricing and portfolio choice. By penalizing the corresponding variational inequality formulation, we transform the stopping problem into a stochastic optimal control problem with two actions. We then randomize controls into Bernoulli distributions and add an entropy regularizer to encourage exploration. We derive a semi-analytical optimal Bernoulli distribution, based on which we devise RL algorithms using the martingale approach established in \cite{jia2022policy}. We establish a policy improvement theorem and \blue{prove the fast convergence of the resulting policy iterations.} We demonstrate the effectiveness of the algorithms in pricing finite-horizon American put options, solving Merton's problem with transaction costs, \blue{and scaling to high-dimensional optimal stopping problems.} In particular, we show that both the offline and online algorithms achieve high accuracy in learning the value functions and characterizing the associated free boundaries.
	
	~\\\textbf{Keywords:} Optimal stopping, reinforcement learning, policy improvement, option pricing, portfolio choice.
\end{abstract}

\section{Introduction}

Optimal stopping involves making decisions on the best timing to exit from (or enter into) an endeavor with an immediate payoff when waiting for a higher payoff in the future is no longer beneficial. Such problems abound in many applications, including finance, physics, and engineering. A prominent example is American-type options where one needs to decide when to exercise the right to buy or sell an underlying stock.

Optimal stopping can be regarded as a control problem in which there are only two control actions: stop (exit) or wait (continue). Yet historically, the optimal stopping theory has been developed fairly separately from that of the general optimal control, owing to the very specific nature of the former, even though both are based on the fundamentally same approaches, namely martingale theory and dynamic programming (DP).
For optimal stopping with a continuous-time diffusion process, DP leads to a variational inequality problem (also known as a free-boundary PDE) that, {\it if} solved, can be used to characterize the stopping/continuation regions in the time--state space and to obtain the optimal value of the problem.

However, there are two issues with the classical approach to optimal stopping, one conceptual and one technical. The first one is that the variational inequality formulation is available only when the model primitives are known and given, while in many applications model primitives are often hard to estimate to the desired accuracy or are simply unavailable.
The second issue is that even if the model is completely known and the variational inequality formulation is available, solving it by traditional numerical PDE methods (e.g., the finite difference method) in high dimensions is notoriously difficult, if not impossible,  due to the curse of dimensionality.


This is where continuous-time reinforcement learning (RL) comes to the rescue. RL has been developed to solve dynamic optimization in partially or completely unknown environments.
This is achieved through interacting with and exploring the environment: the RL agent tries different actions strategically and observes responses from the environment, based on which she learns to improve her action plans. A key to the success of this process is to carefully balance exploration (learning) and exploitation (optimization). RL is also a powerful way to deal with the curse of dimensionality, because it no longer attempts to
solve
dynamic programming equations; rather it employs Monte Carlo and stochastic gradient descent to compute policies and values, which is less prone to the dimensionality complexity (e.g., \citealp{han2018}).

While RL for discrete-time Markov decision processes (MDPs) has been developed for a long time \citep{Sutton2018}, recently the study of RL for controlled continuous-time diffusion processes has drawn increasing attention, starting from \cite{wang2020reinforcement}.
Continuous RL problems are important because many practical problems are naturally continuous in time (and in space). While one can discretize time upfront and turn a continuous-time problem into a discrete-time MDP, it has been experimentally shown (e.g., \citealp{munos2006policy,tallec2019making}) that this
approach is very sensitive to time discretization and performs poorly with small time steps. Equally important, there are more mathematical tools available for the continuous setting that make it possible to establish an overarching and rigorous theoretical foundation for RL algorithms, as demonstrated in a series of papers \cite{jia2022policy,jia2022policygradient,jia2023q}. Indeed, the study on continuous RL may provide new perspectives even for discrete-time MDPs.\footnote{For instance, a theoretical underpinning of \cite{jia2022policy,jia2022policygradient,jia2023q} is the martingale property of certain stochastic processes. It turns out that there is an analogous, and almost trivial, martingality for discrete-time MDPs; see \citet[Appendix C]{jia2022policy}, but such a property had never been mentioned let alone exploited in the literature per our best knowledge.} The general theory has already found applications in finance, including dynamic mean--variance analysis \citep{wang2020continuous,dai2023equilibrium,wu2024reinforcement} and Merton's problem \citep{jiang2022reinforcement,dai2023learning,dai2024learningcost}, to name but a few.

This paper studies optimal stopping in finite time horizons for  (potentially high-dimensional) diffusion processes with unknown model parameters. 
Instead of developing an RL theory for optimal stopping from the ground up parallel to the control counterpart \citep{wang2020reinforcement,jia2022policy,jia2022policygradient,jia2023q},
we first transform the stopping problem into a stochastic control problem, inspired by \cite{dai2007intensity}. Indeed this transformation is motivated by the fact that
the variational inequality problem can be approximated by a ``penalized" PDE, whose solution converges to that of the original variational inequality problem as the penalty parameter approaches infinity \citep{friedman1982,forsyth2002quadratic,liang2007rate,peng2024dpm}. Moreover, the penalized PDE is actually the Hamilton--Jacobi--Bellman (HJB) equation of a stochastic bang-bang control problem. 

Once we have turned the problem into an RL problem for stochastic control, we can apply/adapt all the available results, methods, and algorithms developed for the latter right away. In particular, we take the exploratory formulation of \cite{wang2020reinforcement} involving control randomization and entropy regularization. Because
the control set now contains only two actions (corresponding to stop and continuation), the randomized control is a Bernoulli distribution. To wit, for exploring the environment there is no longer a clear-cut stop or continuation decision. Rather, the agent designs a (biased) coin for each time--state pair and makes the stopping decision based on the outcome of a coin toss.\footnote{Interestingly, in a seemingly different context (a casino gambling model originally proposed by \cite{barberis2012model} featuring behavioral finance theory), \cite{he2017path} and \cite{hu2023path} show that introducing coin-toss for making stopping decisions can on one hand strictly increase the cumulative prospect theory preference value and on the other hand make the problem analytically tractable. However, these papers stop short of explaining why people actually toss coins in reality. Exploration {\it \`a la} RL probably offers one explanation.}
Moreover, we derive a semi-analytical formula for the optimal Bernoulli distribution, which is crucial for policy parameterization and policy improvement in designing our RL algorithms. Built upon the martingale approach established by \cite{jia2022policy}, 
we design two RL algorithms: an offline martingale loss (ML) algorithm and an online TD(0) algorithm. \blue{Moreover, we provide a theoretical guarantee -- a fast convergence result of the policy improvement/iteration.
	
	We then conduct an extensive simulation study, applying these algorithms to several problems, including pricing finite-horizon American put options and  solving Merton's portfolio choice problem with transaction costs. More importantly, we experiment with two {\it high}-dimensional problems: American geometric average put options and optimal stopping for fractional Brownian motions. Through these experiments we demonstrate the effectiveness and efficiency of our reinforcement learning methods and algorithms, particularly their scalability to high-dimensional settings.}


This work is closely related to \cite{dong2023randomized}, the first paper in the literature to study optimal stopping within the continuous-time exploratory framework
of \cite{wang2020reinforcement}. 
However, there are stark differences.
First, we turn the stopping problem into a stochastic control problem using the penalty method, while \cite{dong2023randomized} treats the stopping problem directly by randomizing stopping times with a Cox process. The advantage of our approach, as already noted earlier, is that once in the control formulation, we can immediately apply all the available RL theories for controlled diffusions developed so far, instead of having to develop the corresponding theories for stopping from the get-go.
Second, the entropy regularizers are chosen differently. We adopt the usual Shannon differential entropy, also used by \cite{wang2020reinforcement} and its sequels, while \cite{dong2023randomized} employs 
the so-called unnormalized negentropy. See Section \ref{Subsection: compare} for a more detailed discussion on this point. Finally, in our experiments, we examine the performance of our methods with both online and offline algorithms, while \cite{dong2023randomized} only shows his offline results. 

Another recent relevant paper  is \cite{dianetti2024exploratory}.\footnote{The paper came to our attention shortly after the previous version of our paper was completed and posted.} That paper is also significantly different from ours in several fronts.
First, \cite{dianetti2024exploratory} use randomized stopping times to turn their optimal stopping into a {\it singular} control problem, and then apply the exploratory formulation to the latter which is probably more complex than the former.
By contrast, we do quite the opposite: in solving Merton's problem with transaction costs which is a singular control problem, we turn it into optimal stopping and then into regular control. The reason for us doing so is that singular control involves gradient constraints, which are always hard to treat. In contrast, optimal stopping, along with regular control, is much easier to deal with. That said, it is important to develop an RL theory for singular control (beyond the motivation for solving optimal stopping). In this regard, \cite{dianetti2024exploratory} obtain many results including those of theoretical convergence.
Second, in terms of algorithm design, the value function in \cite{dianetti2024exploratory} is approximated by Monte Carlo thus offline in nature, and their policy improvement involves  the {\it gradient} of the value function that requires very fine grids and a large amount of data to learn. In comparison, we
have both online and offline algorithms, which do not demand computing the aforementioned gradient.
Finally, \cite{dianetti2024exploratory} study stopping in infinite time horizon, while we deal with finite horizons. Besides the consideration of practical applicability, optimal stopping in a finite horizon is significantly  harder than its infinite horizon counterpart. 


This work is also related to a strand of papers on machine learning for optimal stopping (e.g., \citealp{becker2019deep,becker2020pricing,becker2021solving}; \citealp{reppen2022neural}; \citealp{herrera2024optimal}; \citealp{peng2024dpm}).
These papers focus on using deep neural nets to approximate the free boundaries and/or optimal stopping, significantly outside the entropy-regularized, exploratory framework of \cite{wang2020reinforcement} to which the present paper belongs.
\blue{The RL algorithms developed in this paper can also be applied to optimal stopping problems beyond the field of finance. For instance, the Stefan problem in physics, which describes the evolution of the boundary between two phases of a material undergoing a phase change \citep{wang2021deep}, can be formulated as an optimal stopping problem \citep{grigelionis1966stefan}.
In medicine, optimal stopping is relevant to problems  such as radiotherapy treatment, where the goal is to determine when to stop the current treatment scheme \citep{ajdari2019towards}.
Many other engineering problems also involve optimal stopping, such as deciding the best time to stop the production of a natural resource \citep{pham2009}, early stopping of model training in machine learning using Bayesian methods \citep{dai2019bayesian},
and classifying urban reports by
deciding when to stop adding more features once they no longer enhance classification accuracy \citep{liyanage2019automating}.
}

The remainder of the paper is organized as follows. Section \ref{Sec:Optimal stopping} transforms the optimal stopping problem with a finite horizon into a bang-bang stochastic control problem and then presents its exploratory formulation. In Section \ref{Section: RL algorithms}, we propose both the online and offline RL algorithms and provide a policy improvement theorem \blue{with a theoretical convergence rate}. Section \ref{Sec:Simulation} is devoted to simulation studies on pricing an American put option and on an investment problem with transaction costs\blue{, while Section \ref{Sec:Simulation in high-dim} is on two high-dimensional examples}. Finally, Section \ref{Sec:Conclusion} concludes. All the proofs and additional materials are presented in the appendices. 

\section{Optimal stopping in finite time horizon} \label{Sec:Optimal stopping}
We fix a filtered probability space $(\Omega, \mathcal{F}, \bp; \left\{\mathcal{F}_t\right\}_{t\geq 0})$ in which a standard $d$-dimensional Brownian motion $W=(W^{(1)},W^{(2)},\cdots,W^{(d)})$ is defined, along with a fixed finite time $T>0$.
Consider the following optimal stopping problem with a finite horizon
\begin{equation}\label{p1}
\max_{\tau\in[0,T]} \BE{ g(\tau, X_{\tau})}
\end{equation}
for an $n$-dimensional diffusion process $X=(X^{(1)},X^{(2)},\cdots,X^{(n)})$ driven by
\begin{equation}\label{process1}
\dd X^{(i)}_{t}=b_{i}(t,X_t)\dt+\sum_{j}\sigma_{i,j}(t,X_t)\dw^{(j)}_t.
\end{equation}
It is well-known that the HJB equation of the problem \eqref{p1} is given by
\begin{equation}\label{hjbp1}
\begin{cases}
	\max\{v_{t}+\opla v, g-v\}=0,\ (t,x)\in[0,T)\times\R^{n},\\
	v(T,x)=g(T,x),
\end{cases}
\end{equation}
where
\begin{equation}
\opla v(t,x)=\frac{1}{2}\sum_{i,j,k}\sigma_{i,k}(t,x)\sigma_{j,k}(t,x)\frac{\partial^2v }{\partial x_{i}\partial x_{j}}
+\sum_{i} b_{i}(t,x)\frac{\partial v}{\partial x_{i}}.
\end{equation}
The equation \eqref{hjbp1} is equivalent to a set of two inequalities and one equality, referred to as a {\it variational inequality} problem.
%

Throughout this paper, we assume all the coefficients are continuous in $t$ and Lipschitz continuous in $x$. Moreover, $\sigma \sigma'$ is uniformly bounded and positive definite,
where $\sigma=(\sigma_{i,j})_{i,j}$.

\subsection{Penalty approximation}

Consider the following penalty approximation problem for \eqref{hjbp1}:
\begin{equation}\label{hjb2}
\begin{cases}
	v_{t}+\opla v+K(g-v)^{+}=0,\ (t,x)\in[0,T)\times\R^{n},\\
	v(T,x)=g(T,x),
\end{cases}
\end{equation}
where $K$ is a positive scalar, called the \textit{penalty factor}.
The following result states the relationship between \eqref{hjb2} and \eqref{hjbp1}.
\begin{lemma} \label{lemma_K}
The solution of \eqref{hjb2} converges to that of \eqref{hjbp1} as $K\to+\infty$.
\end{lemma}

A proof of this lemma can be found in the literature such as \cite{friedman1982}, 
\cite{liang2007rate}, and \cite{peng2024dpm}. Indeed, one can obtain the following estimate: 
\begin{equation}\label{v_os-v_PI}
0\leq \blue{v^{\text{os}}}(t,x) - v^{K}(t,x) \leq \frac{C}{K}, ~~\forall\; (t,x,K)\in[0,T]\times\R^{n}\times(0,\infty)
\end{equation}
under some technical conditions, where \blue{$v^{\text{os}}$} and $v^{K}$ are respectively the solutions to \eqref{hjbp1} and \eqref{hjb2}, and $C$ is a positive constant independent of $K$.
For reader's convenience we provide a proof of Lemma \ref{lemma_K} in Appendix \ref{Appendix: proof_penalty}.

Lemma \ref{lemma_K} suggests that solving \eqref{hjbp1} reduces to solving \eqref{hjb2} with a sufficiently large $K$. In particular, numerically one can simply solve the latter to get an approximated solution to the former.
We can rewrite \eqref{hjb2} as the following PDE involving a binary optimization term
\begin{equation} \label{hjbp2}
\begin{aligned}
	0 &=v_{t}+\opla v+K(g-v)^{+}
	=v_{t}+\max_{u\in\{0,1\}}\{\opla v-Kuv+Kg(t,x)u\}.
\end{aligned}	
\end{equation}


We are now going to show in the next subsection that this PDE corresponds to some (unconventional) stochastic control problem.

\subsection{Transformation to optimal control}\label{Sec:transform_to_control}

Before we further analyze \eqref{hjbp2}, let us first consider the following general stochastic control problem
\begin{equation}\label{pp1}
v(s,x)=\max_{u\in\setu} J(s,x; u),
\end{equation}
where the controlled $n$-dimensional process $X_t$ follows
\begin{equation}\label{xpp1a}
\dd X_t=\tilde b(t, X_t, u_{t})\dt+\tilde \sigma(t,X_t, u_{t})\dw_t,
\end{equation}
and the objective functional (also called the reward function in the RL literature) is defined as
\begin{equation*}
J(s,x;u)=\BE{\int_s^{T} e^{-\int_{s}^{t}f(r, X_{r}, u_{r})\dr}G(t, X_t, u_{t})\dt
	+e^{-\int_{s}^{T}f(r, X_{r}, u_{r})\dr}H(X_T)\;\bigg|\; X_s=x}.
	\end{equation*}
	Here $\setu$ denotes the set of admissible controls $u:[0,T]\times \Omega\to U$ which are progressively measurable such that \eqref{xpp1a} admits a unique strong solution and $J(s,x;u)$ is finite, while $U\subseteq\R^{m}$ is a given closed control constraint set.
	\revise{In addition, $G:[0,T]\times \R^{n}\times U \to \R$ and $H:\R^{n} \to \R$ represent the running reward and terminal reward terms, respectively, in the objective functional $J$ of the control problem \eqref{pp1}, while $f:[0,T]\times \R^{n}\times U \to \R$ serves as a discount factor.}
	
	
	Note that this is not a conventional stochastic control problem (e.g., \citealp{YZ99}) as it involves control-dependent discount factors. However, we still have the following HJB equation and verification theorem.
	\begin{theorem}[Verification theorem]\label{vf}
Suppose $v$ is a classical solution of
\begin{equation}\label{hjbpp1}
	\begin{cases}
		v_{t}+\max_{u\in U}\left\{\opla_{u} v-vf+G\right\}=0, ~~ (t,x)\in[0,T)\times\R^{n},\\
		v(T,x)=H(x),
	\end{cases}
\end{equation}
where $$\opla_{u} v=\frac{1}{2}\sum_{i,j,k}\tilde\sigma_{i,k}(t,x,u)\tilde\sigma_{j,k}(t,x,u)\frac{\partial^2v }{\partial x_{i}\partial x_{j}}
+\sum_{i} \tilde b_{i}(t,x,u)\frac{\partial v}{\partial x_{i}}.$$
Then it is the optimal value function of problem \eqref{pp1}.
\end{theorem}

\begin{remark}
If $f$ is a constant, then the above result reduces to the classical verification theorem.
One can even establish a similar result when a term like $e^{-\int_{s}^{t}f(r, X_{r}, u_{r})\dr-\int_{s}^{t}g(r, X_{r}, u_{r})\dd W_r}$
is in place of
$e^{-\int_{s}^{t}f(r, X_{r}, u_{r})\dr}$
in the reward functional $J$.
\end{remark}

By \citethm{vf}, the PDE \eqref{hjbp2} is indeed the HJB equation of the following stochastic control problem:
\begin{equation}\label{p2}
v(s,x)=\max_{u_{t}\in\{0,1\}} J(s,x;u),
\end{equation}
where the (uncontrolled) process $X_t$ is driven by \eqref{process1} and
\begin{equation}\label{value_func_control}
J(s,x;u)=\BE{\int_s^{T} Ke^{-K\int_{s}^{t}u_{r}\dr}g(t,X_t)u_{t}\dt
	+e^{-K\int_{s}^{T}u_{r}\dr}g(T,X_T)\;\bigg|\; X_s=x}.
	\end{equation}
	
	Applying \citethm{vf} yields
	\begin{lemma} \label{lemma_control}
Suppose $v$ is a classical solution of \eqref{hjbp2} with terminal condition $v(T,x)=g(T,x)$. Then it is the value function of problem \eqref{p2}.
\end{lemma}

Even though we have established the theoretical dynamic programming result for the non-conventional control problem \eqref{p2}, it is beneficial (to our subsequent analysis) to present an alternative conventional formulation. This is achieved by introducing an additional state.

Recalling  that $X_t$ is driven by \eqref{process1}, we introduce a new controlled state equation $\dd R_{t}=-KR_{t}u_{t}\dt$ or $R_t=R_se^{-K\int_{s}^{t}u_{r}\dr}$ for $t\geq s$.
Let
\begin{equation}\label{p2b}
\barv(s,x,r)=\max_{u_{t}\in\{0,1\}} \barJ(s, x, r; u),
\end{equation}
where
\begin{equation*}
\barJ(s, x, r; u)=\BE{\int_s^{T} KR_{t}g(t,X_t)u_{t}\dt+R_{T}g(T,X_T)\;\bigg|\; X_s=x,R_{s}=r}.
\end{equation*}
Then clearly problem \eqref{p2} has the optimal value $v(s,x)=\barv(s,x,1).$

As is well known, the HJB equation for the classical control problem \eqref{p2b} is given by
\begin{equation}\label{hjbp2b}
\begin{cases}
	\displaystyle{	\barv_{t}+\max_{u\in\{0,1\}} \Big[\opla \barv- Kru\frac{\partial \barv}{\partial r}+Krg(t,x)u\Big]=0,\ (t,x,r)\in[0,T)\times\R^{n}\times(0,\infty),}\\
	\barv(T,x,r)=rg(T,x),~~
	\barv(t,x,0)=0.
\end{cases}
\end{equation}

Clearly, for $r>0$, $\barJ(s, x, r; u)=r \barJ(s, x, 1; u)$;
so $\barv(s,x, r)=r\barv(s,x,1)=rv(s,x). $
Substituting this into \eqref{hjbp2b}, we get, for $ (t,x)\in[0,T)\times\R^{n}$,
\begin{equation}\label{hjbp2b2}
\begin{cases}
	v_{t}+\max\limits_{u\in\{0,1\}} \Big[\frac{1}{2}\sum_{i,j,k}\sigma_{i,k}(t,x)\sigma_{j,k}(t,x)\frac{\partial^2v }{\partial x_{i}\partial x_{j}}
	+\sum_{i} b_{i}(t,x)\frac{\partial v}{\partial x_{i}}- Kuv+Kg(t,x)u\Big]=0,\qquad\\
	v(T,x)=g(T,x),
\end{cases}
\end{equation}
which recovers \eqref{hjbp2}. In this sense, the two problems \eqref{p2b} and \eqref{p2} are equivalent.



Finally it is clear from \eqref{hjbp2} that the control/action $u=1$ (resp. 0) for \eqref{p2b} corresponds to stop (resp. continue) for the original stopping problem.

\subsection{RL formulation}

Now we are ready to present the RL formulation for  problem \eqref{p2b} by considering
the entropy-regularized, exploratory version of the problem, which was
first introduced by \cite{wang2020reinforcement} and then studied in many continuous-time RL examples \citep{wang2020continuous, jiang2022reinforcement, jia2022policygradient, jia2023q, dong2023randomized, dai2023learning}. A special feature of the current problem is that, because the action set has only two points, 0 and 1, the exploratory control process $\pi$ is a Bernoulli-distribution valued process with $\BP{u_{t}=1}=\pi_{t}=1-\BP{u_{t}=0}$. When exercising such a control, at any $t$ the agent flips a coin having head with probability $\pi_t$ and tail with probability $1-\pi_t$, and stops if head appears and continues otherwise.

With a slight abuse of notation, henceforth we still use $(X,R)$ to denote the exploratory states. Since $X$ is not controlled, it is still driven by \eqref{process1} but $R$ now satisfies
\begin{equation}\label{dynamics_R}
\dd R_{t}=-K R_{t}\pi_{t}\dt.
\end{equation}
Then the entropy-regularized exploratory counterpart of problem \eqref{p2b} is
\begin{equation}\label{problem_high_dim}
\barv^{\boldsymbol{\Pi}}(s,x,r )=\max_{\pi} \barJ^{\pi}(s, x, r),
\end{equation}
where the superscript $\boldsymbol{\Pi}$ is a reminder that we are currently considering an exploratory problem with control randomization and entropy regularization, and the maximization is over all Bernoulli-distribution valued processes taking values in $\{0,1\}$. The value function is 
\begin{equation}\label{Jpi}
\barJ^{\pi}(s, x, r)=\BE{\int_s^{T} R_{t}\big[Kg(t,X_t)\pi_{t} - \lam \BH(\pi_{t})\big]\dt+R_{T}g(T,X_T) \;\bigg|\; X_s=x,R_{s}=r},
\end{equation}
where
\begin{equation}
\BH(\pi)=\pi \log \pi+(1-\pi)\log\left(1-\pi\right).
\end{equation}
In the above, $-\BH(\pi)$ is the differential entropy of the Bernoulli distribution $\pi$, and $\lam>0$ is the \textit{exploration weight} or \textit{temperature parameter} representing the trade-off between exploration and exploitation.

Note that the formulation  \eqref{Jpi} does not follow strictly the general formulation in \cite{wang2020reinforcement} and those in the subsequent works, because the running entropy regularization term in \eqref{Jpi}  is $-R_t\BH(\pi_t)$ instead of $-\BH(\pi)$. There are two reasons behind this modification. First, 
{$\barv(t,x,0)=0$}
in \eqref{hjbp2b} for both current and future time--state pairs $(t,x)$
motivates us to expect that
{$\barv^{\boldsymbol{\Pi}}(t,x,0)=0$} holds as well, which would be untrue if we took $-\BH(\pi)$ as the regularizer.
Second, with the modification we can reduce the dimension of the problem as follows.
Taking the expression of $R_{t}$ into $\barJ^{\pi}(s, x, r)$ to get
\begin{equation}\label{p3}
v^{\boldsymbol{\Pi}}(s,x)=\max_{\pi_{t}\in[0,1]} J^{\pi}(s, x),
\end{equation}
where $	J^{\pi}(s,x)$ defined below does not rely on $R_{t}$ anymore:
\begin{equation}\label{value_func_control_entropy}
J^{\pi}(s,x)=\BE{\int_s^{T} e^{-K\int_{s}^{t}\pi_{v}\dv}\big[K g(t,X_t) \pi_{t}-\lam \BH(\pi_{t})\big]\dt+e^{-K\int_{s}^{T}\pi_{v}\dv}g(T,X_T) \;\bigg|\; X_s=x }.
\end{equation}
Note that we have the following relationships:
\begin{equation}\label{linear_homo_value_func}
\barJ^{\pi}(s, x, r)=rJ^{\pi}(s,x) \ \text{ and }\ v^{\boldsymbol{\Pi}}(s,x)=\frac{1}{r}\barv^{\boldsymbol{\Pi}}(s, x, r)=\barv^{\boldsymbol{\Pi}}(s, x, 1)\ \text{ for all }\ r>0.
\end{equation}
Meanwhile, the optimal strategies of \eqref{problem_high_dim} and \eqref{p3} are identical, which will be shown shortly.
Consequently, in the RL algorithms in the next section, to learn the optimal value function $v^{\boldsymbol{\Pi}}(s,x)$ of the problem \eqref{p3}, we will start from the problem \eqref{problem_high_dim} with initial state $R_0=1$.

By \citethm{vf}, the HJB equation of the problem \eqref{p3} is
\begin{equation}\label{hjbp3}
v_{t}+\opla v+\max_{\pi\in [0,1]}\big\{-K\pi v+Kg \pi - \lam (\pi\log \pi+(1-\pi)\log (1-\pi))\big\}=0,\ (t,x)\in[0,T)\times\R^{n}.	
\end{equation}
Since $\pi\mapsto -K\pi v+Kg \pi -\lam (\pi\log \pi+(1-\pi)\log (1-\pi))$
is strictly concave, the unique maximizer $\pi^{*}$ in the above is determined by
the first order condition:
$-Kv+Kg-\lam (\log \pi^{*}- \log (1-\pi^{*}))=0$, or
\begin{equation}\label{equ update pi}
\pi^{*}
=\frac{1}{1+\exp(\lam^{-1}K(v-g))}.
\end{equation}
%
%
Using this, the HJB equation \eqref{hjbp3} reduces to
\begin{equation}\label{hjb5}
\exp(-\lam^{-1} (v_{t}+\opla v))=\exp(\lam^{-1} (-Kv+Kg))+1.
\end{equation}

By \citethm{vf}, we have
\begin{lemma}
Suppose $v$ is a classical solution of \eqref{hjb5} satisfying the terminal condition $v(T,x)=g(T,x)$. Then it is the value function of the problem \eqref{p3}.
\end{lemma}

%

Following the same procedure, the HJB equation of  \eqref{problem_high_dim} is
\begin{equation}
\begin{aligned}
	\barv_{t}+\opla \barv+\max_{\bar{\pi}\in [0,1]}r\Big\{-K\bar{\pi} \barv_r+Kg \bar{\pi} -\lam (\bar{\pi}&\log \bar{\pi}+(1-\bar{\pi})\log (1-\bar{\pi}))\Big\}=0,
\end{aligned}	
\end{equation}
where    $(t,x, r)\in[0,T)\times\R^{n}\times [0,\infty)$,
and the unique maximizer in above is
\begin{equation}
\bar{\pi}^{*}=\frac{1}{1+\exp(\lam^{-1}K(\barv_r-g))}.
\end{equation}
By \eqref{linear_homo_value_func}, $\barv_r^{\boldsymbol{\Pi}}(s,x,r)=v^{\boldsymbol{\Pi}}(s,x)$,
therefore, $\bar{\pi}^{*}(s,x,r)=\pi^{*}(s,x)$,
which shows that the optimal strategies of the problems \eqref{problem_high_dim} and \eqref{p3} are the same.

\blue{
We end this section by providing an error analysis arising from the penalty approximation (in terms of $K$) and the
exploratory framework (in terms of $\lambda$).
It is immediate that the optimal value function $v^{K}$ of the penalty approximation (which satisfies \eqref{hjbp2}) also satisfies \eqref{hjbp3}  with $\lam=0$, while the optimal value function $v^{\boldsymbol{\Pi}}$ of the exploratory control problem \eqref{p3} satisfies \eqref{hjbp3} with a positive $\lambda$.
By the maximum principle, it is straightforward that $v^{K}\leq v^{\boldsymbol{\Pi}}$.

Next, we provide a lower bound of $v^{K}$.
Consider the function $\tilde{v}^{\lam}:=v^{K}+\lam(T-t)\log 2$, which coincides with $v^{\boldsymbol{\Pi}}$ when $t=T$.
Substituting $\tilde{v}^{\lam}$ into the LHS of \eqref{hjbp3} yields:
\begin{align*}
	& \tilde{v}^{\lam}_t + \mathcal{L}\tilde{v}^{\lam} + \max_{\pi \in [0,1]} \left\{ -K\pi \tilde{v}^{\lam} + Kg\pi - \lambda \big( \pi \log \pi + (1-\pi) \log(1-\pi) \big) \right\}\\
	= &\ v^{K}_t + \mathcal{L}v^{K} -\lam\log 2\\&\quad+ \max_{\pi \in [0,1]} \left\{ -K\pi v^{K} - K\pi\lam(T-t)\log 2 + Kg\pi - \lambda \big( \pi \log \pi + (1-\pi) \log(1-\pi) \big) \right\}
	\\
	\leq &\  v^{K}_t + \mathcal{L}v^{K} + \max_{\pi \in [0,1]} \left\{ -K\pi v^{K}  + Kg\pi  \right\} + \max_{\pi \in [0,1]}\left\{- K\pi\lam(T-t)\log 2\right\} \\
	&\quad + \max_{\pi \in [0,1]}\left\{- \lambda \big( \pi \log \pi + (1-\pi) \log(1-\pi) \big)-\lam\log 2\right\}
	\leq\  0,
\end{align*}
where we make use of the facts that $v^{K}_t + \mathcal{L}v^{K} + \max_{\pi \in [0,1]} \left\{ -K\pi v^{K}  + Kg\pi  \right\}=0$ and $-\big( \pi \log \pi + (1-\pi) \log(1-\pi) \big)\leq \log 2$. Thus, again by the maximum principle, we have that $\tilde{v}^{\lam}\geq v^{\boldsymbol{\Pi}}$, i.e., $v^{K}\geq v^{\boldsymbol{\Pi}}-\lam(T-t)\log 2$.

\begin{theorem}\label{Thm:Errors}
	For any fixed penalty factor $K$, we have
	\begin{equation*}
		v^{\boldsymbol{\Pi}}-\lam(T-t)\log 2\leq v^{K}\leq v^{\boldsymbol{\Pi}},
	\end{equation*}
	where $v^{K}$ and $v^{\boldsymbol{\Pi}}$ satisfy \eqref{hjbp2} and \eqref{hjbp3}, respectively.
	Combined with \eqref{v_os-v_PI}, we have 
	\begin{equation*}
		-\lam (T-t)\log 2\leq  v^{\text{os}} - v^{\boldsymbol{\Pi}}\leq \frac{C}{K},
	\end{equation*}
	where $v^{\text{os}}$ satisfying \eqref{hjbp1} is the optimal value function of the original optimal stopping problem \eqref{p1}, and $C$ is a positive constant independent of $\lambda$ and $K$.
\end{theorem}

As a consequence, as we increase the penalty factor $K$ to infinity and decrease the temperature parameter $\lam$ to zero, $v^{\boldsymbol{\Pi}}$ converges uniformly to $v^{\text{os}}$.

}

\section{Reinforcement learning algorithms}\label{Section: RL algorithms}

We discretize the time horizon $\left[0, T\right]$ into a series of equally spaced time intervals $\left[t_{l-1}, t_l\right]$, $1\leq l \leq L$ with $t_l=l\Delta t$. Given a control policy $\pi$, approximate its value function of problem \eqref{p3} at time $t_l$, i.e., $J^{\pi}\left(t_l,\cdot\right)$, by a neural network (NN) with parameter $\theta_l\in\mathbb{R}^{N_{\theta_l}}$, where $N_{\theta_l}$ represents the dimension of $\theta_l$, for $l=0,\ldots, L-1$.\footnote{For problems with special structures, we can exploit them to directly parameterize the value function instead of resorting to the NN, which could reduce the complexity greatly. For example, \cite{wang2020continuous} and \cite{jiang2022reinforcement} use analytical forms of the value functions for parameterizing a pre-committed dynamic mean-variance problem and a dynamic terminal log utility problem, respectively.}
We denote by $V_{\theta_l}(\cdot)$ this approximated value function at time $t_l$.
Besides, at $t_L=T$ the value function is simply $V_{\theta_L}(\cdot)=g(T,\cdot)$.
Finally, \eqref{linear_homo_value_func} yields the corresponding approximation of the value function of the problem \eqref{problem_high_dim}.

A typical RL task consists of two steps: policy evaluation and policy improvement. Starting from an initial (parameterized) policy $\pi^{(0)}$, we need to learn the value function under $\pi^{(0)}$ to get an evaluation of this policy. The next step is to find a direction to update the policy to $\pi^{(1)}$ such that the value function under $\pi^{(1)}$ is improved compared to that under $\pi^{(0)}$. We repeat this procedure until optimality or near-optimality is achieved.

The second step, policy improvement, is based on the following theoretical result (which is analogous to \eqref{equ update pi}).

\begin{theorem}[Policy improvement theorem]\label{policy improvement}
Given an admissible policy $\pi$ together with its value function $J^{\pi}\left(\cdot,\cdot\right)$, let
\begin{equation}\label{tilde_pi}
	\tilde{\pi}(t,x)=\frac{1}{1+\exp\left(\lam^{-1}K\left(J^{\pi}\left(t,x\right)-g(t,x)\right)
		\right)}.
\end{equation}
Then $J^{\pi}\left(t,x\right)\leq J^{\tilde{\pi}}\left(t,x\right)$ for all $(t,x)\in [0,T]\times \R^{n}$.
	\end{theorem}
	
	So for the current problem, the policy improvement can be computed {\it analytically} provided that $J^{\pi}\left(\cdot,\cdot\right)$ is accessed for any $\pi$. Thus, solving the problem boils down to the first step, policy evaluation.\footnote{\blue{Note that this shortcut is specific to the current problem and the approach we take.}}
	For that step, we take the martingale approach for continuous-time RL developed in \cite{jia2022policy}.
	
	First, we consider an offline algorithm. Following \cite{jia2022policy}, denote
	\begin{equation*}
\mathcal{M}_t:=\barJ^{\pi}(t, X_t, R_t)+\int_0^{t} R_{s}\big[Kg(s,X_s)\pi_{s}- \lam \BH(\pi_{s})\big]\ds,
\end{equation*}
which is a martingale. 
Noting $\barJ^{\pi}(T, X_T, R_T)=R_{T}g(T,X_T)$, \eqref{linear_homo_value_func} and the NN approximation, we get
\begin{equation*}
\begin{aligned}
	\mathcal{M}_T-\mathcal{M}_t &=\barJ^{\pi}(T, X_T, R_T)- \barJ^{\pi}(t, X_t, R_t)+\int_t^{T} R_{s}\big[Kg(s,X_s)\pi_{s}- \lam \BH(\pi_{s})\big]\ds\\
	&=R_{T}g(T,X_T) - R_t J^{\pi}(t, X_t)+\int_t^{T} R_{s}\big[Kg(s,X_s)\pi_{s}- \lam \BH(\pi_{s})\big]\ds\\
	&=R_{T}g(T,X_T) - R_t V_{\theta}(X_t)+\int_t^{T} R_{s}\big[Kg(s,X_s)\pi_{s}- \lam \BH(\pi_{s})\big]\ds.
\end{aligned}
\end{equation*}
%
The martingale loss function is thus
\begin{equation}
\begin{aligned}\label{martingale_loss}
	\text{ML}\left(\left\{\theta_l\right\}_{l=0}^{L-1}\right)
	&=\frac{1}{2}\mathbb{E}\left[\int_0^T|\mathcal{M}_T-\mathcal{M}_t|^2dt\right]
	\approx:\frac{1}{2}\mathbb{E}\bigg[ \sum_{l=0}^{L-1}G_l^2 \Delta t\bigg],
\end{aligned}
\end{equation}
where
\begin{equation}
G_l=R_{T}g(T,X_T) - R_{t_l} V_{\theta_l}(X_{t_l})+\sum_{j=l}^{L-1} R_{t_j}\big[Kg(t_j,X_{t_j})\pi_{t_j}- \lam \BH(\pi_{t_j})\big]\Delta t. \label{ML_component}
\end{equation}
For states at discretized time points, we adopt the forward Euler scheme as follows,
\begin{equation}
X_{t_{l+1}}=X_{t_l}+b(t_l,X_{t_l})\Delta t+\sigma(t_l,X_{t_l})\Delta W\quad \text{with } \Delta W\sim\mathcal{N}(0,\Delta t),\label{discrete X}
\end{equation}
and
\begin{equation}
R_{t_{l+1}}=R_{t_l}\left(1-K\pi_{t_l}\Delta t\right)\quad \text{for }l=0,\ldots,L-1.\label{discrete R}
\end{equation}

It is obvious that to minimize the martingale loss, the state trajectories in the whole time horizon are required, which leads to the
{\it offline} RL Algorithm \ref{algorithm_ML}.
During training, as long as we can collect the realized outcomes of the payoff function $g$, referred to  as the {\it reward signals}  from the environment, the algorithm is applicable even if we do not know the functional form of $g$.\footnote{In some situations, we may also be able to learn  $g$ through parameterizations. For instance, in the next section we will show that when $g$ is the payoff function corresponding to the early exercise premium of an American put option, whose exact value is dependent on the unknown market volatility, an accurate approximation of $g$ is available upon training by the martingale approach.}

For {\it online} learning, we apply the martingale orthogonality condition $\mathbb{E}\left[\int_0^T\xi_t \text{d} \mathcal{M}_t\right]=0$
for any test function $\xi\in L_{\mathcal{F}}^2\left([0,T],\mathcal{M}\right)$, again proposed by \cite{jia2022policy}, where
\begin{align*}
&\quad L_{\mathcal{F}}^2\left([0,T],\mathcal{M}\right)\\&:=\left\{\kappa=\left\{\kappa_t, 0\leq t\leq T\right\}: \kappa \text{ is $\mathcal{F}_t$-progressively measurable and } \mathbb{E}\int_0^T |\kappa_t|^2 \text{d} \langle \mathcal{M}\rangle_t < \infty \right\}.
\end{align*}
In particular, a TD(0)-type algorithm results  from choosing $\xi_t=\frac{\partial \barJ^{\pi}(t, X_t, R_t)}{\partial \theta}$. The corresponding online updating rule is
$$\theta \leftarrow \theta+\alpha \frac{\partial \barJ^{\pi}(t, X_t, R_t)}{\partial \theta} \text{d}\mathcal{M}_t,$$
where $\alpha$ is the learning rate.
For actual implementation where time is discretized, we have at time $t_l$:
\begin{equation*}
\begin{aligned}
	&\quad \frac{\partial \barJ^{\pi}(t_l, X_{t_l}, R_{t_l})}{\partial \theta_l} \text{d} \mathcal{M}_{t_l} \\
	&\approx\frac{\partial \left(R_{t_l}V_{\theta_l}(X_{t_l})\right)}{\partial \theta_l} \left\{R_{t_{l+1}}V_{\theta_{l+1}}(X_{t_{l+1}}) - R_{t_l}V_{\theta_l}(X_{t_l})+R_{t_l}\left[Kg\left(t_l,X_{t_l}\right)\pi_{t_l}-\lambda\BH(\pi_{t_l})\right]\Delta t\right\}\\
	& \propto \frac{\partial V_{\theta_l}(X_{t_l})}{\partial \theta_l}\left\{\frac{R_{t_{l+1}}}{R_{t_l}}V_{\theta_{l+1}}(X_{t_{l+1}}) - V_{\theta_l}(X_{t_l})+
	\left[Kg(t_l,X_{t_l})\pi_{t_l} - \lambda \BH(\pi_{t_l})\right]\Delta t \right\}.
\end{aligned}
\end{equation*}
Consequently, the updating rule in our online algorithm is
\begin{equation}\label{update_TD0}
\theta_l \leftarrow \theta_l+\alpha \frac{\partial V_{\theta_l}(X_{t_l})}{\partial \theta_l}\left\{\frac{R_{t_{l+1}}}{R_{t_l}}V_{\theta_{l+1}}(X_{t_{l+1}}) - V_{\theta_l}(X_{t_l})+
\left[Kg(t_l,X_{t_l})\pi_{t_l} - \lambda \BH(\pi_{t_l})\right]\Delta t\right\}
\end{equation}
for $l=0,\ldots,L-1$.
Algorithm \ref{algorithm_TD0} presents the online learning procedure.

\begin{algorithm}\label{algorithm_ML}
\caption{Offline ML}
\textbf{Inputs:} initial states $x_0$ and $r_0=1$, penalty factor $K$, time horizon $T$, time step $\Delta t$, number of time intervals $L$, training batch size $M$, number of training steps $N$, learning rate $\alpha$, temperature parameter $\lambda$ and a good representation of the payoff function $g(\cdot,\cdot)$.\\
\textbf{Learning procedure:}\\
Initialize the neural networks $V_{\theta_l}(\cdot)$ for $l=0,\ldots,L-1$.\\
\For{$n=1,\dots,N$}{
	Generate $M$ state trajectories of $\left\{X_{t_l}\right\}_{l=0}^L$ with $X_0=x_0$ by \eqref{discrete X}.\\
	Calculate the strategy $\left\{\pi\left(t_l,X_{t_l}\right)\right\}_{l=0}^L$ according to $\pi\left(t_l,X_{t_l}\right)=\frac{1}{1+\exp\left(\frac{K}{\lambda}\left(V_{\theta_l}(X_{t_l})-g(t_l,X_{t_l})\right)\right)}$.\\
	Generate $M$ state trajectories of $\left\{R_{t_l}\right\}_{l=0}^L$ with $R_0=r_0$ by \eqref{discrete R}.\\
	
	Calculate the martingale loss \eqref{martingale_loss}.\\
	Update $\left\{\theta_l\right\}_{l=0}^{L-1}$ by taking one step stochastic gradient descent with respect to the martingale loss with learning rate $\alpha$.\\		
}
\end{algorithm}
\begin{algorithm}\label{algorithm_TD0}
\caption{Online TD(0)}
\textbf{Inputs:} initial states $x_0$ and $r_0=1$, penalty factor $K$, time horizon $T$, time step $\Delta t$, number of time intervals $L$, training batch size $M$, number of training steps $N$, learning rate $\alpha$, temperature parameter $\lambda$ and a good representation of the payoff function $g(\cdot,\cdot)$.\\
\textbf{Learning procedure:}\\
Initialize the neural networks $V_{\theta_l}(\cdot)$ for $l=0,\ldots,L-1$.\\
\For{$n=1,\dots,N$}{
	Generate $M$ state trajectories with the identical initial states $X_0=x_0$ and $R_0=r_0$.\\
	\For{$l=0,\dots,L-1$}{
		Generate $X_{t_{l+1}}$ by \eqref{discrete X} for all $M$ state trajectories.\\
		Calculate the strategy $\pi\left(t_l,X_{t_l}\right)$ according to $\pi\left(t_l,X_{t_l}\right)=\frac{1}{1+\exp\left(\frac{K}{\lambda}\left(V_{\theta_l}(X_{t_l})-g(t_l,X_{t_l})\right)\right)}$.\\
		Generate $R_{t_{l+1}}$ by \eqref{discrete R} for all $M$ state trajectories.\\	
		Update $\theta_l$ by stochastic approximation \eqref{update_TD0} with learning rate $\alpha$.\\
	}		
}
\end{algorithm}

\blue{
Finally, we provide some convergence result as a theoretical guarantee of the proposed RL algorithm, assuming that the policy evaluation can be carried out  precisely.\footnote{\blue{Here we ignore the impact of time discretization in the presented convergence result to avoid distraction from the central ideas of the paper. For a  convergence analysis involving the effects of time discretization and sampling in randomization, see Section 3.2.2 in \cite{dai2025datadriven} for Merton's problem with power utility in a Black-Scholes market, where there is only one unknown model primitive. For more general problems with more complex structures, see a recent paper \cite{jia2025accuracy}. This is still an on-going active research topic in the realm of RL for controlled diffusion processes.}}
Starting from a given policy $\pi^0$ and its value function $J^{\pi^0}$, consider a sequence of policies $\left\{\pi^n\right\}_{n=1}^{\infty}$ and the corresponding value functions $\left\{J^{\pi^n}\right\}_{n=1}^{\infty}$, where
\begin{align}\label{Eq: PolicyImproveIteration}
	\pi^{n+1}(t,x) = \frac{1}{1+\exp\left(\lambda^{-1}K\left(J^{\pi^n}(t,x)-g(t,x)\right)\right)}.
\end{align}
It follows from  the Feynman--Kac formula that
\begin{align}\label{Eq: PolicyEvaluateFeynmanKac}
	\partial_t J^{\pi^n}+\opla J^{\pi^n} + H\left(t,x,\pi^n,J^{\pi^n}\right) = 0,\quad J^{\pi^n}\left(T,x\right)=g(T,x),
\end{align}
where
$	H\left(t,x,\pi,J^{\pi}\right) = K\pi(t,x)\left(g(t,x)-J^{\pi}(t,x)\right)-\lambda\BH(\pi(t,x))$
is as defined in the proof of Theorem \ref{policy improvement}.

\begin{theorem}[Convergence rate of policy iteration]\label{convergence of policy improvement}
	Given any initial guess $J^{\pi^0}\in C_b(\mathbb{R})$ for the optimal exploratory control problem \eqref{p3} under an initial admissible policy $\pi^0$, generate the sequence $\left\{\pi^n,J^{\pi^n}\right\}_{n=1}^{\infty}$ alternatingly by \eqref{Eq: PolicyImproveIteration} and \eqref{Eq: PolicyEvaluateFeynmanKac}.	
	Suppose the payoff function $g$ is uniformly bounded. Then, for any fixed penalty factor $K$ and temperature parameter $\lam$, we have
	\begin{align*}
		\left\| v^{\boldsymbol{\Pi}} - J^{\pi^{n+1}} \right\|_{\infty}
		\leq   \frac{(2KT)^n}{n!}  2[(K\left\|g\right\|_{\infty}+\lam\log 2)T+\left\|g\right\|_{\infty}].
	\end{align*}
	In particular, 
	it follows that $J^{\pi^n}$ converges to $v^{\boldsymbol{\Pi}}$ as $n\rightarrow\infty$.
\end{theorem}

}

\section{Simulation studies: American put options}\label{Sec:Simulation}

In this section, we present simulation experiments in which an RL agent learns how to exercise finite-horizon American put options. 
A put option can be formulated as the following optimal stopping problem under the risk-neutral measure
\begin{equation}\label{put}
\max_{\tau\in[0,T]} \BE{ e^{-\rho\tau} \left(\Gamma-X_{\tau}\right)^+},
\end{equation}
where the stock price follows a one-dimensional geometric Brownian motion:
$	\dd X_t=\rho X_t\dt+\sigma X_t \dw_t$
with $\rho$ being the risk-free rate and $\sigma$ being the volatility of the stock price, and the (time-invariant) payoff function is
$	g(x)=\left(\Gamma-x\right)^+$
with $\Gamma$ being the strike price.
In the RL setting, the volatility $\sigma$ is unknown to the agent.

Due to the presence of the exponential time discounting in the objective function, \eqref{ML_component} and \eqref{update_TD0} in Algorithms \ref{algorithm_ML} and \ref{algorithm_TD0} need to be respectively modified to (refer to Section 5.1 in \citealp{jia2022policy}):
\begin{equation*}
G_l=e^{-\rho T}R_{T}g(X_T) - e^{-\rho t_l}R_{t_l} V_{\theta_l}(X_{t_l})+\sum_{j=l}^{L-1} e^{-\rho t_j} R_{t_j}\big[Kg(X_{t_j})\pi_{t_j}- \lam \BH(\pi_{t_j})\big]\Delta t,
\end{equation*}
and
\begin{equation*}
\theta_l \leftarrow \theta_l+\alpha \frac{\partial V_{\theta_l}(X_{t_l})}{\partial \theta_l}\left\{\frac{R_{t_{l+1}}}{R_{t_l}}V_{\theta_{l+1}}(X_{t_{l+1}}) - V_{\theta_l}(X_{t_l})+
\left[Kg(X_{t_l})\pi_{t_l} - \lambda \BH(\pi_{t_l})\right]\Delta t-\rho V_{\theta_l}(X_{t_l})\right\}.
\end{equation*}

In our experiments, we use the same parameters as in \cite{becker2021solving} and \cite{dong2023randomized} for the simulator, i.e., the initial stock price $x_0=40$, strike price $\Gamma=40$, risk-free rate $\rho=0.06$, volatility $\sigma=0.4$ (but unknown to the agent), and maturity time $T=1$.
We use the implicit finite difference method to solve the penalty approximation of the HJB equation of the optimal stopping problem, namely \eqref{hjb2},
and obtain the omniscient option price which equals 5.317, as well as the free boundary $X_{f}(t)$ which is shown as the blue line in Figure \ref{fig: stopping_classifications_ML_TD0_by_VE}.

Finally, Table \ref{table: benchmark values} lists some key (hyper)parameter values in our experiments.

\begin{table}[htbp]
\begin{center}
	\caption{Benchmark values in RL training. }
	\label{table: benchmark values}
	\footnotesize
	\begin{tabular}{l|c|c}		
		\noalign{\hrule height 0.7pt}
		\hline
		& Notation & Benchmark value\\	
		\hline
		Number of time mesh grids &$L$& 50 \\
		\hline
		Penalty factor & $K$ & 10\\
		\hline
		Temperature parameter & $\lambda$ & 1\\
		\hline
		Learning rate & $\alpha$ & 0.01\\
		\hline
		Training batch size & $M$ & $2^{10}$\\
		\hline
		Testing batch size & $M_{\mathrm{test}}$ & $2^{18}$\\
		\hline
		Number of training steps & $N$ & 1000 or 5000\\		
		\noalign{\hrule height 0.7pt}
		\hline
	\end{tabular}
\end{center}
\end{table}


\subsection{Learning early exercise premium}\label{Subsection: learn early exercise premium}


Early exercise premium is the difference in value between an American option and its European counterpart. In our setting, it is defined as 
\begin{equation}\label{v_tilde}
\tilde{v}(t,x):=v(t,x)-V_E(t,x), \vspace{-5pt}
\end{equation}
where 
$V_E(t,x)$ is the value of the European put option with the same strike. The associated payoff function is
\begin{equation}\label{g_tilde}
\tilde{g}(t,x):=g(x)-V_E(t,x).
\end{equation}

Instead of learning the value of the American put option, we set out to learn the value of the early exercise premium. An advantage of doing so is that it removes the singularity of the terminal condition, which improves the learning performance around the maturity time. 

In the following subsections, we will learn the early exercise premium $\tilde{v}$ by solving the related entropy-regularized problem with payoff $\tilde{g}$. However, the unknown model primitive $\sigma$ renders an unknown $V_E$ embedded in $\tilde{g}$; so the problem has an unknown payoff function. We will apply policy evaluation to learn an approximated $\tilde{g}^{\phi}$; refer to
Appendix \ref{learn early exercise premium} for details.


The architecture of the fully connected feedforward NNs used in our study is the same as that in \cite{dong2023randomized}: At each decision time point $t\in\{t_l: l=0,1,\ldots,L-1\}$, we construct one neural network, which consists of one input layer, two hidden layers with 21 neurons followed by an ReLU activation function and one output layer.
Following \cite{becker2021solving}, at time $t_l$ we take the inputs to be the state and the approximated payoff value, i.e., $X_{t_l}$ and $\tilde{g}^{\phi}(t_l, X_{t_l})$. The output $w_{\theta_l}$ of the NN at time $t_l$ is the difference between the early exercise premium and the associated payoff function, i.e., $\tilde{v}(t_l,X_{t_l})-\tilde{g}(t_l,X_{t_l})$, which is equal to $v(t_l,X_{t_l}) - g(t_l,X_{t_l})$ in theory by \eqref{v_tilde} and \eqref{g_tilde}.
Write 
\begin{equation}\label{output_NN}
w_{\theta_l}\left(X_{t_l}, \tilde{g}^{\phi}(t_l, X_{t_l})\right) :=V_{\theta_l}(X_{t_l}) - \tilde{g}^{\phi}(t_l, X_{t_l}),
\end{equation}
where $V_{\theta_l}(\cdot)$ is the NN approximator for the value function of the exploratory problem regarding $\tilde{v}$, namely problem \eqref{v_tilde randomized} at time $t_l$.
Finally, after the input layer and the output layer and before every activation layer, we add a batch normalization layer to facilitate training.

\subsection{Learned value functions}\label{Subsection: learned value functions}
With the  parameter values given in Table \ref{table: benchmark values}, we implement training for 1000 steps by the offline ML Algorithm \ref{algorithm_ML} and for 5000 steps by the online TD(0) Algorithm \ref{algorithm_TD0}.
These are carried out within the PyTorch framework.
Each algorithm spends about 0.5s for one step training on a Windows 11 PC with an Intel Core i7-12700 2.10 GHz CPU and 32GB RAM by CPU computing only.

Figure \ref{fig: value_functions_compare} shows the learning results by Algorithms \ref{algorithm_ML} and \ref{algorithm_TD0} regarding the exploratory value functions at six selected decision time points $t\in\{t_l: l=0, 10, 20, 30, 49\}=\{0, 0.2, 0.4, 0.6, 0.8, 0.98\}$.
In each panel, we plot five lines with the stock price ranging from zero to twice the initial stock price.
The solid red line (with label ``Option'') shows the theoretical (omniscient) values of 
$v(t,x)-g(x)$ derived by finite difference.
The dotted solid blue, orange and green lines (with label ``Dong (2024)'', ``ML'' and ``TD(0)'') are respectively the outputs $v-g$ from the NN by the TD error algorithm in \cite{dong2023randomized} and by the ML and TD(0) algorithms in this paper.
The dotted yellow vertical line (with label ``$S^{*}$'') refers to the specific stock price that separates the stopping and continuation regions at each time, which is also derived by finite difference.

\begin{figure}[htb]
\centering
\subfigure
{\includegraphics[scale=0.33]{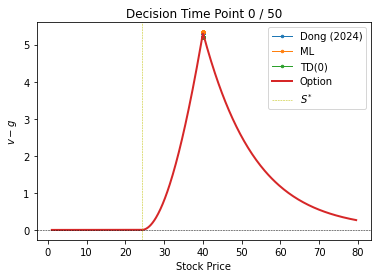}
}
\quad
\subfigure
{\includegraphics[scale=0.33]{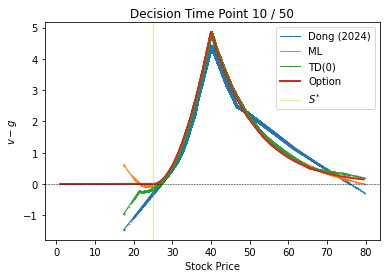}
}
\quad
\subfigure
{\includegraphics[scale=0.33]{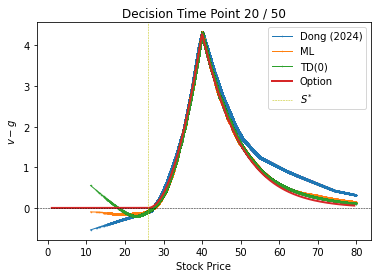}
}
\quad
\subfigure
{\includegraphics[scale=0.33]{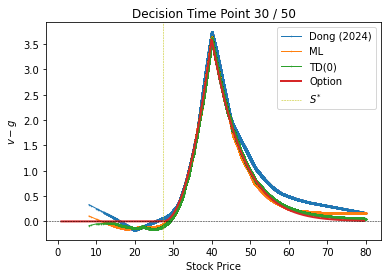}
}
\quad
\subfigure
{\includegraphics[scale=0.33]{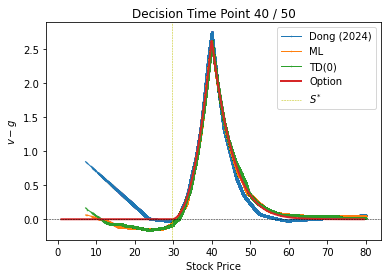}
}
\quad
\subfigure
{\includegraphics[scale=0.33]{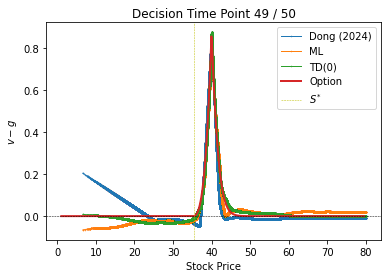}
}
\caption{Learning results of $v-g$ under the TD algorithm in \cite{dong2023randomized}, the offline ML algorithm and the online TD(0) algorithm in this paper at six selected decision time points: $\{0, 10, 20, 30, 40, 49\}$-th discretized time points. The solid red lines are the theoretical values of $v-g$ and the dotted yellow vertical lines specify the stock price on the theoretical free boundary, both derived by finite difference. The training parameters of Dong (2024) are specified in Section \ref{Subsection: compare}.}
\label{fig: value_functions_compare}
\end{figure}

As observed, all the algorithms result in values close to the oracle value, especially when the stock price is above $S^{*}$. The results of \cite{dong2023randomized} have relatively high errors when approaching the terminal time, which is the main motivation for learning early exercise premium instead in this paper. Except for extremely low stock prices, which are rarely reached by our simulated stock price trajectories leading to too few data for adequate learning, the yellow dotted vertical lines well separate the positive values from the non-positive ones.
This result is desired when we construct the execution policy in the next subsection.


\subsection{Execution policy}\label{Subsection: execution policy}

After having learned the exploratory value functions with
well-trained NNs, we can further obtain the stopping probabilities by \eqref{equ update pi}. However, when it comes to {\it real} execution regarding stopping or continuing, it does not seem reasonable to actually toss a biased coin. To wit, randomization is necessary for {\it training}, but not for {\it execution}. For the latter, one way is to use the mean of the randomized control \citep{dai2023learning}, which is, however, inapplicable in the current 0--1 optimal stopping context. For example, a mean of 0.8 is outside of the control value set $\{0, 1\}$. 

Recall that for the classical (non-RL) optimal stopping problem of the American put option, the stopping region is characterized by $\mathcal{S}=\left\{(t,x): v(t,x)=g(x)\right\}$
while the continuation region is characterized by $\mathcal{C}=\left\{(t,x): v(t,x)>g(x)\right\}$.
Thus, it is reasonable to base the sign of $w_{\theta_l}=V_{\theta_l}-\tilde{g}^{\phi}$ in \eqref{output_NN} to make an exercise decision at time $t_l$.
Theoretically, by \eqref{equ update pi}, when the penalty factor $K$ goes to infinity, $\pi^{*}$ tends to 1 and 0 when $w_{\theta_l}<0$ and $w_{\theta_l}>0$, respectively.
However, $K$ is finite in our implementation, $\pi^{*}$ in \eqref{equ update pi} is more likely to be neither 0 nor 1; but 0.5 is an appropriate threshold to divide between $w_{\theta_l}<0$ and $w_{\theta_l}>0$.
In Figure \ref{fig: stopping_classifications_ML_TD0_by_VE}, we randomly choose 100 stock price trajectories from the testing dataset, and classify the learned stopping strategies into two groups based on either $\pi^*\geq0.5$ or $\pi^*<0.5$. When the former occurs the agent chooses to exercise the option, which is represented by the green triangle; 
otherwise the agent does not exercise, represented by the red circle.
Besides, the blue line corresponds to the free boundary of this American put option $X_f(t)$.
As Figure \ref{fig: stopping_classifications_ML_TD0_by_VE} shows, after sufficient training, both algorithms effectively classify almost all the exercise/non-exercise time--price pairs.

Next, we calculate the accuracy rate of classification at each time point as 
$\qquad$ 
$\frac{\#\left\{\text{correctly classified sample points}\right\}}{M_{\mathrm{test}}}.$
It turns out
that both algorithms achieve accuracy rates of higher than $95\%$ for all discretized time points.\footnote{One never exercises an American put option when the stock price is greater than the strike price; so actually we only need to study the points below the black dotted lines in Figure \ref{fig: stopping_classifications_ML_TD0_by_VE}. With this in mind the accuracy rates achieved are even higher.}
\begin{figure}[htb]
\centering
\subfigure[Random initialization]
{
	\includegraphics[scale=0.35]{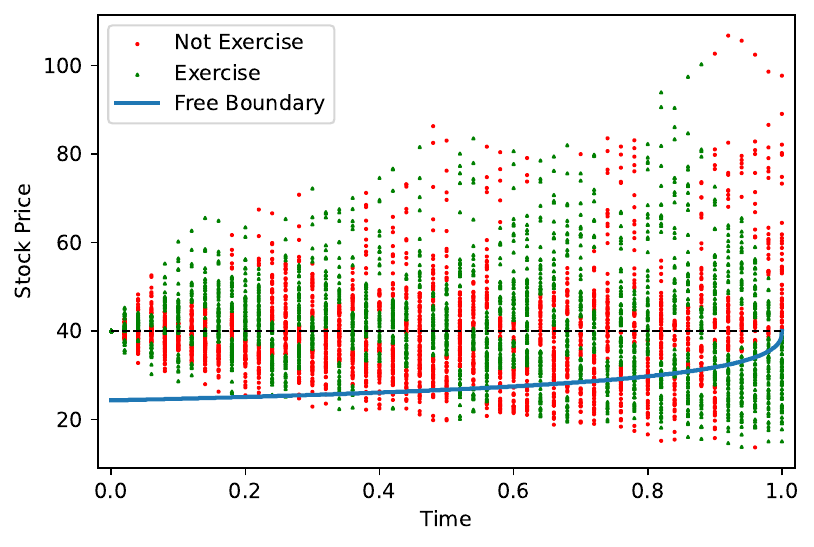}		
}
\quad
\subfigure[By ML with 1000 training steps]
{
	\includegraphics[scale=0.35]{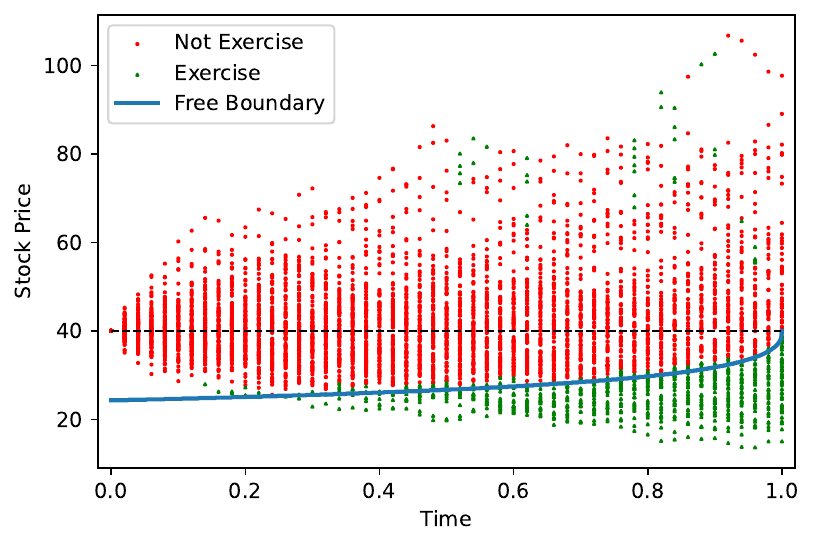}
}
\quad
\subfigure[By TD(0) with 5000 training steps]
{
	\includegraphics[scale=0.35]{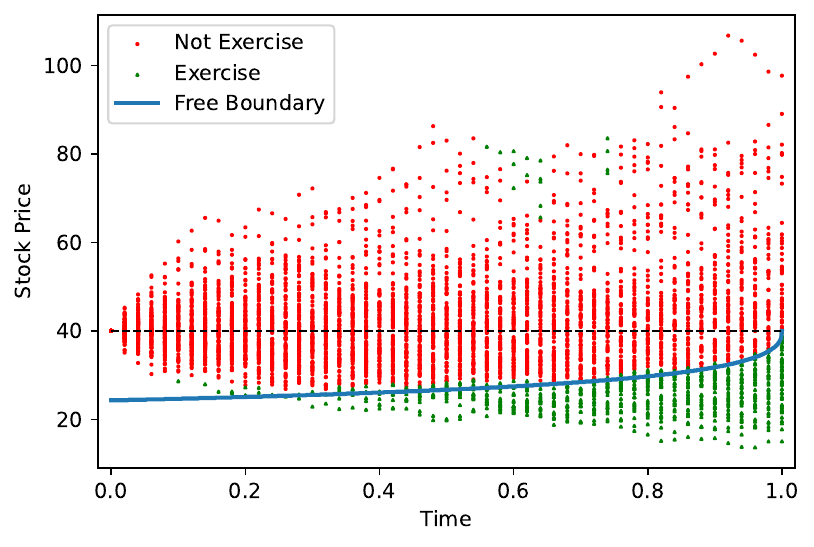}
}
\caption{Stopping strategies learned from the two algorithms. The blue line is the theoretical free boundary of the American put option derived by finite difference. In each panel, there are 100 stock price trajectories with 50 discretized time points. The red circles correspond to the cases for which the stopping probability is less than half, while the green triangles to the cases for which the stopping probability is larger than or equal to half.} 
\label{fig: stopping_classifications_ML_TD0_by_VE}
\end{figure}



\subsection{Option price}\label{Subsection: option price}

With the execution policy determined in Section \ref{Subsection: execution policy}, how to compute the option price?
In this subsection, we introduce two ways to do it based on the learned exploratory value functions in Section \ref{Subsection: learned value functions}.

The first way is to calculate the value of the original optimal stopping problem with the stopping time determined by the learned exploratory value functions in Section \ref{Subsection: learned value functions}, namely
\begin{equation}\label{option_price_by_stopping}
\mathcal{P}_{\mathrm{stopping}} :=\mathbb{E}\left[e^{-\rho \tau}g(X_{\tau})\right],
\end{equation}
where 
\begin{equation}
\tau :=\inf \left\{t_l: V_{\theta_l}(X_{t_l})\leq \tilde{g}^{\phi}(t_l, X_{t_l}),\ l\in\left\{0,1,...,L\right\} \right\} \wedge T
\end{equation}
with the convention that the infimum of an empty set is infinity.

The second way is to calculate the value \eqref{value_func_control} at the initial state $s=0, x=x_0$ under the execution policy determined in Section \ref{Subsection: execution policy}, namely
\begin{equation}\label{option_price_by_control}
\mathcal{P}_{\mathrm{control}}:=\BE{\int_s^{T} Ke^{-K\int_{s}^{t}u_{r}\dr}g(t,X_t)u_{t}\dt
+e^{-K\int_{s}^{T}u_{r}\dr}g(T,X_T)\;\bigg|\; X_s=x},
\end{equation}
where 
\begin{equation}
u(t,x)=1 \mbox{ if } V_{\theta_l}(x)\leq \tilde{g}^{\phi}(t, x), \mbox{ and 0 otherwise, for $(t,x)\in [0,T)\times \mathbb{R}$.}
\end{equation}
Recalling that the omniscient option price is 5.317, we define the relative error of a learned price $\mathcal{P}$  to be $\frac{|\mathcal{P}-5.317|}{5.317}$.  
We calculate the relative error on the testing dataset every 10 training steps.
The learning curves of option price from the two methods described above are plotted in Figure \ref{fig: learn_curve_ML_TD0_by_VE}.

\begin{figure}[htb]
\centering
\subfigure[$\mathcal{P}_{\mathrm{stopping}}$]
{\includegraphics[scale=0.45]{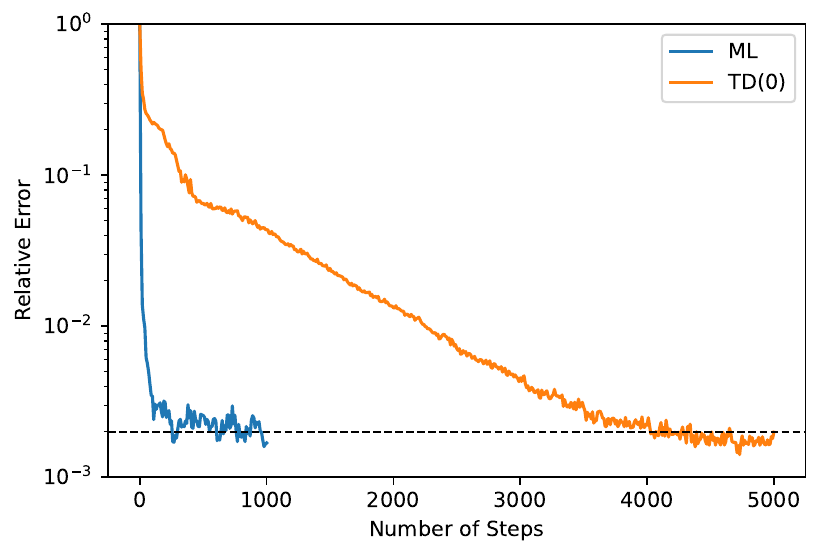}	
}
\quad
\subfigure[$\mathcal{P}_{\mathrm{control}}$]
{\includegraphics[scale=0.45]{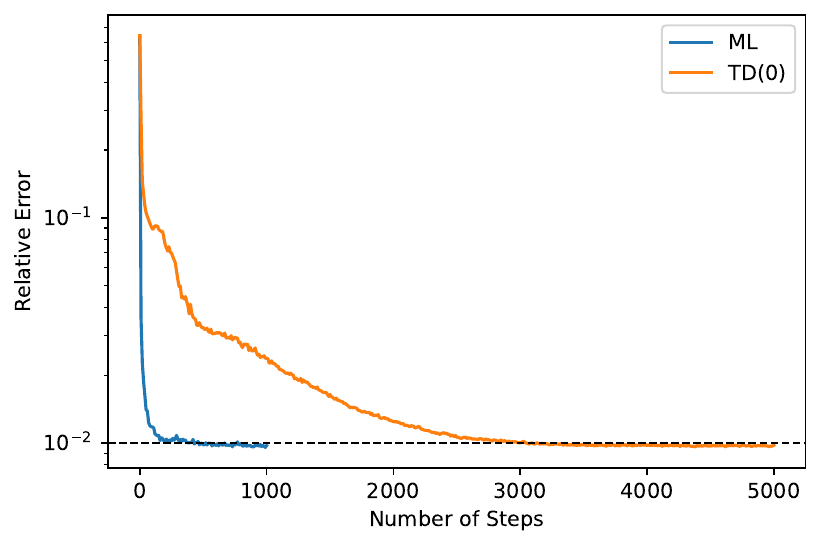}
}
\caption{Learning curves of option price $\mathcal{P}_{\mathrm{stopping}}$ defined in \eqref{option_price_by_stopping} and $\mathcal{P}_{\mathrm{control}}$ defined in \eqref{option_price_by_control} by offline ML (blue line) and online TD(0) (orange line) algorithms.
We train the ML algorithm for 1000 steps and TD(0) algorithm for 5000 steps. Values of model parameters used for training are specified in Table \ref{table: benchmark values}.}
\label{fig: learn_curve_ML_TD0_by_VE}
\end{figure}


First, we find that the ML algorithm has faster convergence than the online TD(0). This is because the latter needs to update parameters at every small time steps. 
Second, both algorithms eventually achieve very low relative errors (about $0.2\%$ and $1\%$ on average for $\mathcal{P}_{\mathrm{stopping}}$ and $\mathcal{P}_{\mathrm{control}}$, respectively). 
Thus, the RL algorithms prove to be effective in terms of learning the option price without knowing the true value of the volatility $\sigma$.

We observe that
the relative errors in the second method converge to the level about $1\%$. This is due to the finite value of the penalty factor $K=10$ chosen in Table \ref{table: benchmark values}. In the next subsection, we will show that as $K$ increases, the relative error can be further reduced.


\subsection{Sensitivity analysis}
In this subsection, we carry out a sensitivity analysis on the impacts of the values of the penalty factor $K$ and the temperature parameter $\lambda$ on learning.

\subsubsection{Penalty factor}\label{Subsec:penalty_factor}

We examine the effect of penalty factor on the learned option price in this subsection. 
First,
{from a numerical implementation perspective,} it follows from \eqref{discrete R} that $0<K\leq\frac{1}{\Delta t}=50$ in our setting, {since the process $R_t$ cannot become negative}.
We pick $K\in\left\{1, 5, 10, 20, 30, 40, 50\right\}$ and keep the other model parameter values unchanged as in Table \ref{table: benchmark values}. Figure \ref{fig: learn_curve_ML_TD0_by_VE_as_K_varies} shows the learning curves of option prices with a combination of the two algorithms and the two alternative ways.
As expected, the higher the value of $K$, the less the relative error achieved by both algorithms. However, when $K$ increases to a certain level, a further increase no longer helps decrease the error, in which case the sampling error dominates.

\begin{figure}[htb]
\centering
\subfigure[$\mathcal{P}_{\mathrm{stopping}}$ by ML]
{\includegraphics[scale=0.38]{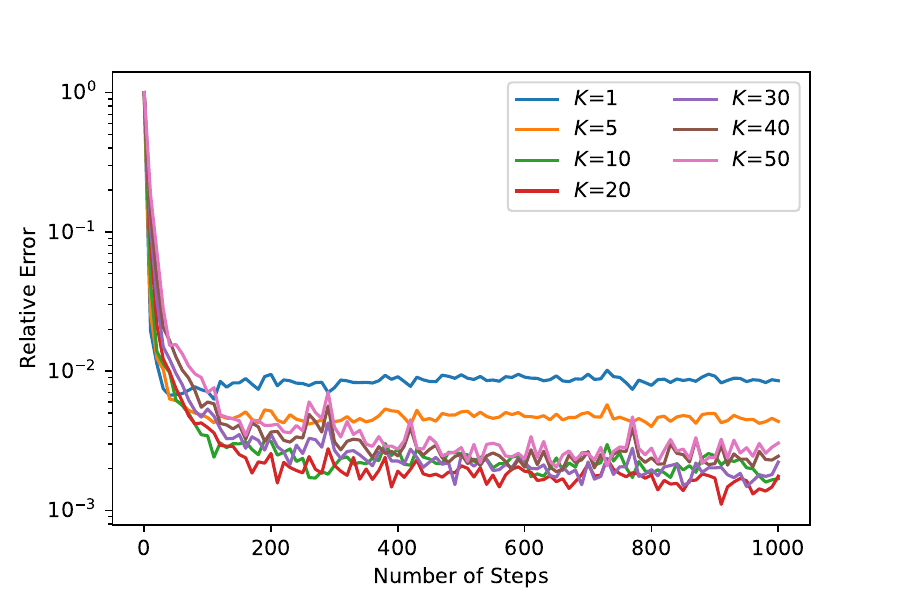}
}
\subfigure[$\mathcal{P}_{\mathrm{stopping}}$ by TD(0)]
{\includegraphics[scale=0.38]{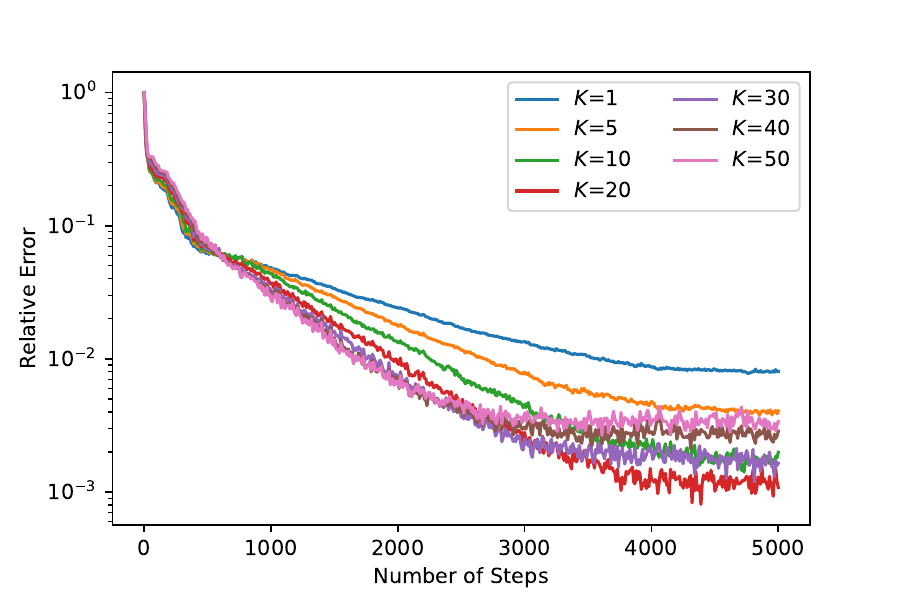}
}
\subfigure[$\mathcal{P}_{\mathrm{control}}$ by ML]
{\includegraphics[scale=0.38]{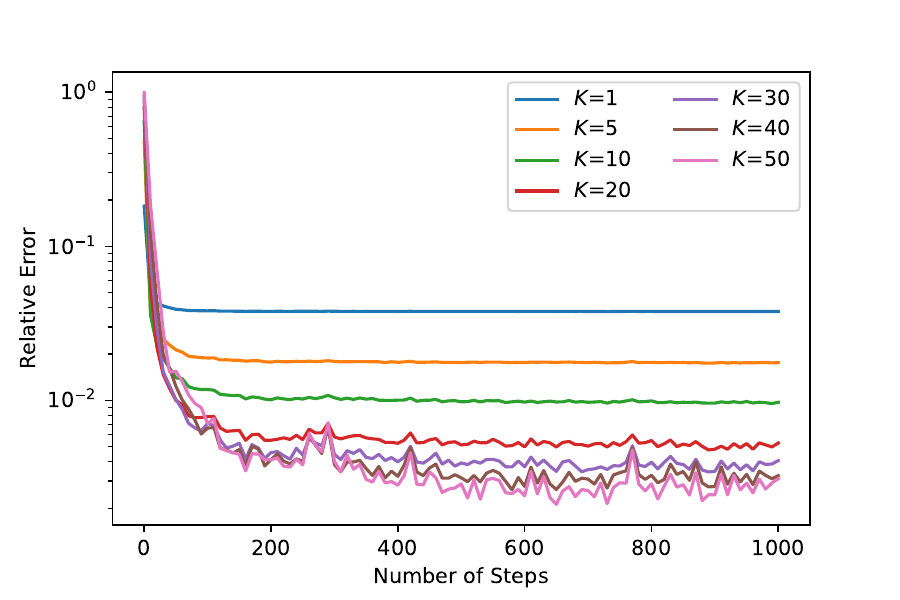}
}
\subfigure[$\mathcal{P}_{\mathrm{control}}$ by TD(0)]
{\includegraphics[scale=0.38]{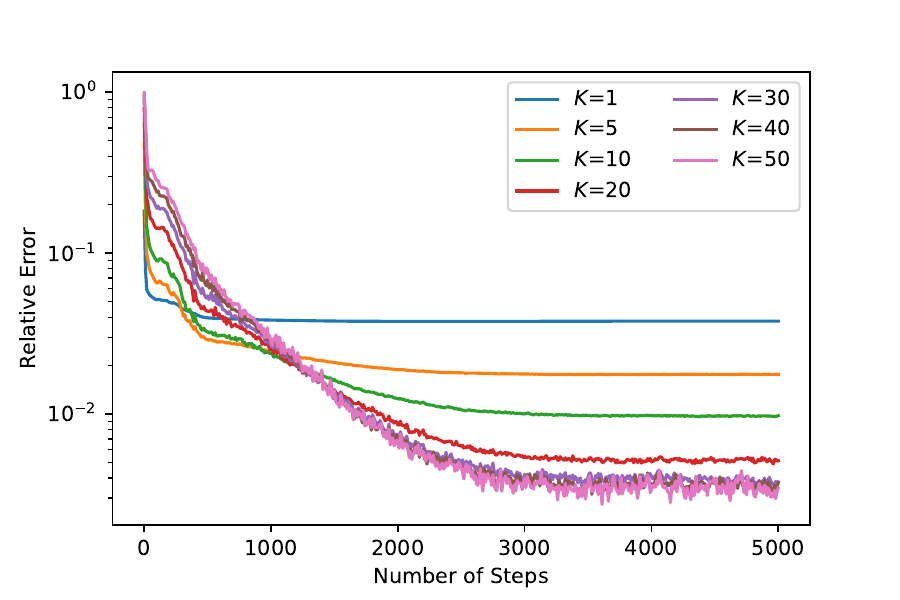}
}
\caption{Learning curves of option prices $\mathcal{P}_{\mathrm{stopping}}$ defined in \eqref{option_price_by_stopping} and $\mathcal{P}_{\mathrm{control}}$ defined in \eqref{option_price_by_control} by offline ML and online TD(0) algorithms. In each panel, the penalty factor $K$ takes value from the set $\{1, 5, 10, 20, 30, 40, 50\}$ while other parameter values are set in Table \ref{table: benchmark values}. We train the ML algorithm for 1000 steps and TD(0) algorithm for 5000 steps. In all cases, the relative errors decrease in $K$.}
\label{fig: learn_curve_ML_TD0_by_VE_as_K_varies}
\end{figure}

\subsubsection{Temperature parameter}

The temperature parameter controls the weight on exploration, and thus is important for learning. 
Figure \ref{fig: learn_curve_ML_TD0_by_VE_as_lam_varies_fix_K_10}
shows the learning curves of option prices when $\lambda\in\left\{0.1, 0.5, 1, 3, 5, 10\right\}$, while keeping the other model parameter values as in Table \ref{table: benchmark values}.\footnote{In
Appendix \ref{Appendix: K=50}, we also present the results for a higher penalty factor $K=50$.}
We observe that in general, the smaller the temperature parameter, the more accurate but the lower the learning speed for convergence, which is especially evident with the TD(0) algorithm.\footnote{Similar results are also noted in \cite{dong2023randomized}.}
An implication of this observation is that if the computing resource is limited so the number of training steps is small, we can increase $\lambda$ to accelerate convergence.
Intriguingly,
in Figure \ref{fig: learn_curve_ML_TD0_by_VE_as_lam_varies_fix_K_10}, the errors are not monotonic in $\lambda$ when $\lambda\leq 1$ in the two cases using $\mathcal{P}_{\mathrm{stopping}}$. We conjecture that it is because the effect of $\lambda$ is dominated by the sampling error when $\lambda$ is sufficiently small. However, overall, the differences are around $0.3\%$ for $\lambda\in\{0.1, 0.5, 1\}$, which are sufficiently small.

\begin{figure}[htb]
\centering
\subfigure[$\mathcal{P}_{\mathrm{stopping}}$ by ML]
{\includegraphics[scale=0.38]{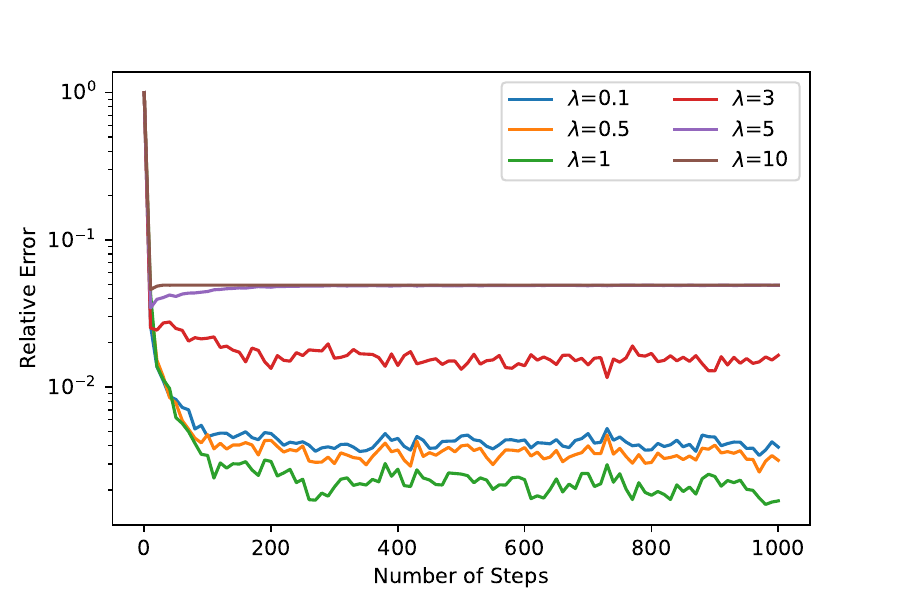}
}
\subfigure[$\mathcal{P}_{\mathrm{stopping}}$ by TD(0)]
{\includegraphics[scale=0.38]{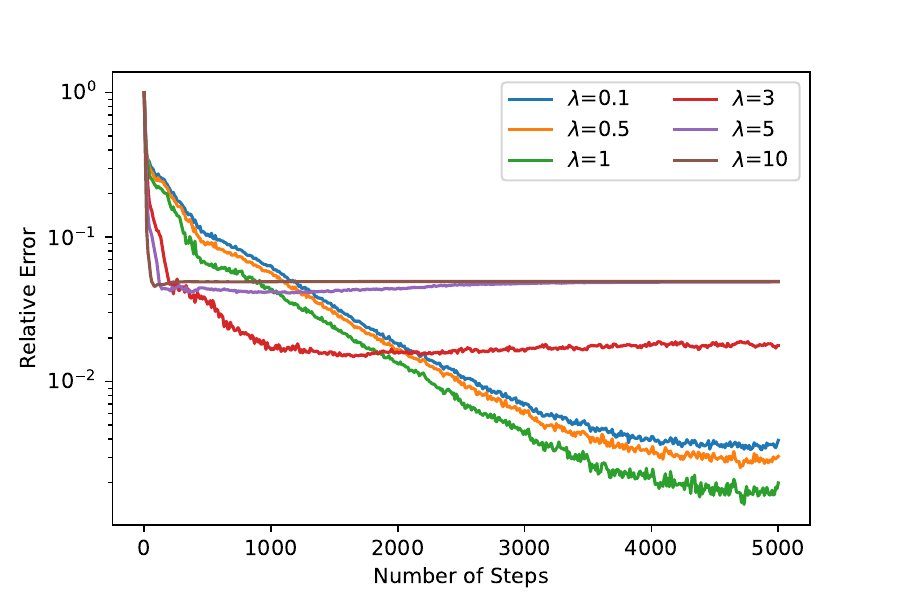}
}
\subfigure[$\mathcal{P}_{\mathrm{control}}$ by ML]
{\includegraphics[scale=0.38]{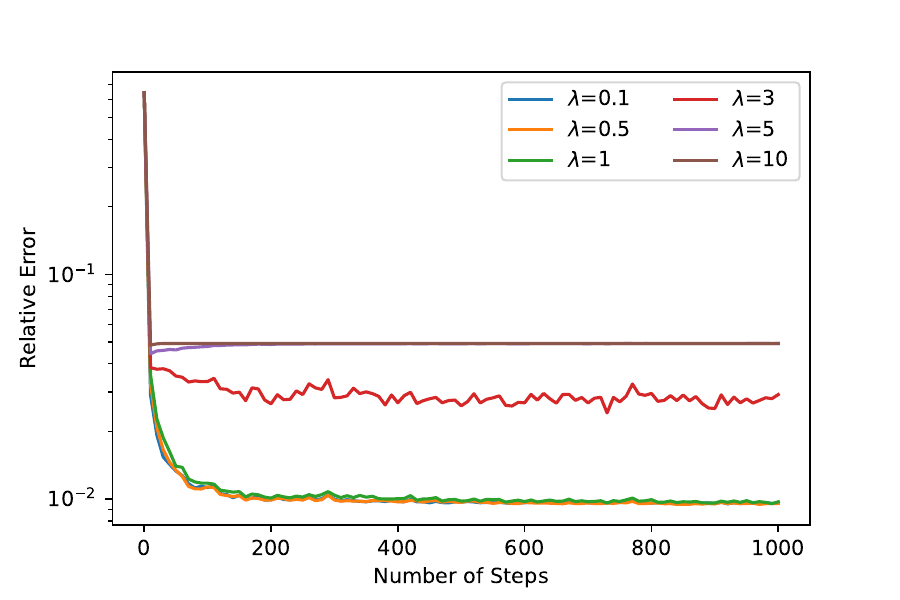}
}
\subfigure[$\mathcal{P}_{\mathrm{control}}$ by TD(0)]
{\includegraphics[scale=0.38]{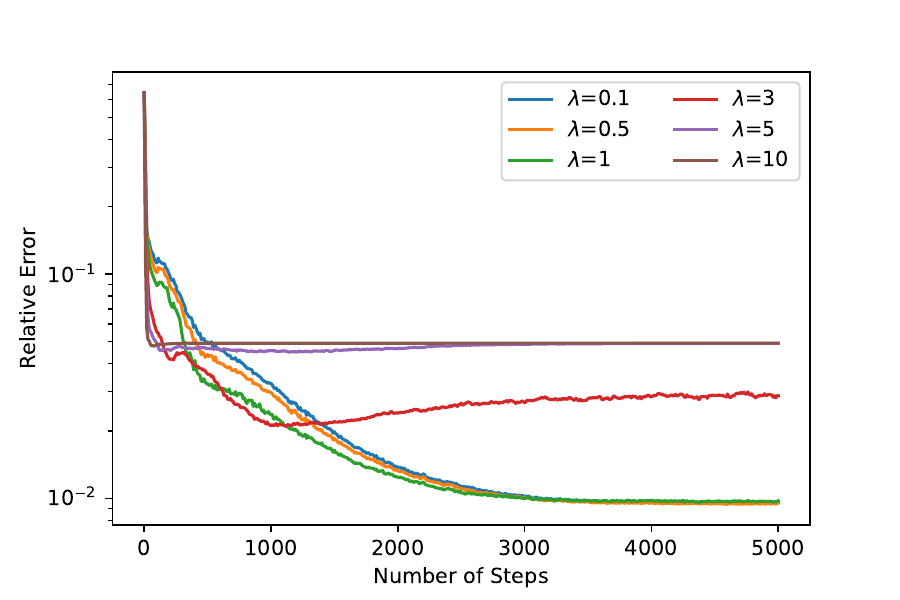}
}
\caption{Learning curves of $\mathcal{P}_{\mathrm{stopping}}$ defined in \eqref{option_price_by_stopping} and $\mathcal{P}_{\mathrm{control}}$ defined in \eqref{option_price_by_control} by offline ML and online TD(0) algorithms. In each panel, the penalty factor $K$ is fixed to 10 and the temperature parameter $\lambda$ takes value from the set $\{0.1, 0.5, 1, 3, 5, 10\}$ while other parameter values are chosen by Table \ref{table: benchmark values}. We train the ML algorithm for 1000 steps and TD(0) algorithm for 5000 steps.}
\label{fig: learn_curve_ML_TD0_by_VE_as_lam_varies_fix_K_10}
\end{figure}


\subsection{Comparison with Dong (2024)}\label{Subsection: compare}

Both \cite{dong2023randomized} and this paper study RL for optimal stopping, but there are major differences as discussed in the introduction section.
In particular, there is a considerable technical difference in terms of the entropy terms introduced in the respective RL objectives. It turns out that, when formulated in our setting with the penalty factor $K$, the decision variable $\tilde{\pi}$ in \cite{dong2023randomized} is the product of $K$ and the probability of stopping in our paper, i.e., $\tilde{\pi}=K\pi$.
The entropy used in \cite{dong2023randomized} is
$R(\tilde{\pi}) :=\tilde{\pi} - \tilde{\pi} \log\tilde{\pi}=R(K\pi)=K\pi - K\pi\log(K\pi),$
whereas we take the differential entropy
$-\BH(\pi)=-\pi\log \pi - (1-\pi)\log(1-\pi).$
\cite{dong2023randomized} explains that $R(\tilde{\pi})$ is proposed by assuming $\tilde{\pi}\Delta t$ to be the probability of stopping in a $\Delta t$-length time interval and then integrating the differential entropy
$-\BH(\tilde{\pi}\Delta t)=-\tilde{\pi}\Delta t\log\Delta t+\left(\tilde{\pi}-\tilde{\pi}\log\tilde{\pi}\right)\Delta t+o(\Delta t)$
over the whole time horizon while dropping the first divergence term.
In other words, it is not the exact differential entropy of the randomized strategy as in our paper.

Next, we compare the performances of the algorithms between the two papers.
We use $\mathcal{P}_{\mathrm{stopping}}$ defined in \eqref{option_price_by_stopping} for computing option price, as it is the one taken by \cite{dong2023randomized}.
For a fair comparison, we use the TD error employed in \cite{dong2023randomized} as the criterion for performance and we learn the option price instead of the early exercise premium.
We implement training for 1000 steps for both algorithms and repeat the training for 10 times.
Because the scales of the entropy terms used in the two papers are considerably different, we train the two algorithms for different combinations of temperature parameter $\lam\in\left\{10, 1, 0.1, 0.01, 0.001, 0.0001\right\}$ and learning rate $\alpha\in\left\{0.1, 0.01, 0.001, 0.0001\right\}$. The other parameters are kept the same as in Table \ref{table: benchmark values} and the simulated stock price samples are the same for both algorithms.
We then check the learning speeds of the two algorithms. Since the theoretical option price is 5.317, we examine the learning speeds for all the cases of $(\lambda,\alpha)$ in which the learned option prices are at least 5.29 for a meaningful comparison. We find that our algorithm is always faster than \cite{dong2023randomized}'s.
Figure \ref{fig:compare_learn_curve} is an illustration with $(\lambda, \alpha)=(0.001, 0.01)$ (these are the model parameters corresponding to the solid blue line in Figure \ref{fig: value_functions_compare}) in Dong's algorithm and $(\lambda, \alpha)=(1,0.1)$ in our algorithm.
The figure shows that both algorithms achieve very low relative errors ($0.39\%$ for our method and $0.44\%$ for Dong's method) after 1000 training steps, but our algorithm is around 200 steps faster in convergence.

\begin{figure}[htb]
\centering
\includegraphics[scale=0.45]{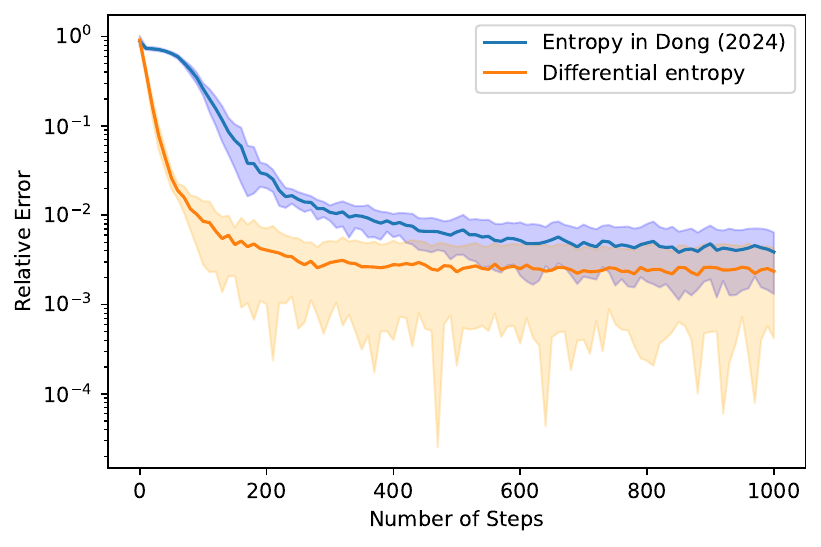}
\caption{Learning curves of option price from our algorithm and \cite{dong2023randomized}'s.
We repeat the training for 10 times to calculate the standard deviations of the learned option prices, which are represented as the shaded areas. The width of each shaded area is twice the corresponding standard deviation.
Note that the vertical axis uses a logarithmic scale.}
\label{fig:compare_learn_curve}
\end{figure}

%

The simulation study reported so far is for an American option, whose price is evaluated under the {\it risk-neutral} probability. However, the underlying stock price is in general {\it not} observable in the risk-neural world; hence pricing American options is not the best example for illustrating a data-driven approach such as RL.
Here, we present another application to demonstrate the effectiveness of our method in which the physical probability measure instead of the risk-neutral one is used (and thus all the relevant data are observable).
Specifically,
\cite{dai2024learningcost} consider the problem of learning an optimal investment policy to maximize a portfolio's terminal log return in the presence of proportional transaction costs, and turn it into an equivalent Dynkin game which involves optimal stopping decisions of two players. The latter is naturally formulated under the physical
measure, and suitable for RL study under the frameworks of both \cite{dong2023randomized} and this paper; refer to
Appendix \ref{RL_Dynkin_penalty}  for details. \cite{dai2024learningcost} extend the RL algorithm of \cite{dong2023randomized} to learn an optimal investment policiy with transaction costs. With the same model parameters as in \cite{dai2024learningcost} and a penalty factor $K=5$ --- which is an appropriate choice as specified in Section \ref{Subsec:penalty_factor} when the time horizon $\left[0,T\right]$ with $T=1$ is equally divided into 10 time intervals (so that $\Delta t=1/10$, and there are two stopping times in the Dynkin game) --- we learn how to trade in the Black--Scholes market under the penalty formulation. We then compare the learning results for the free boundaries and the average log return when applying the learned strategies over a 20-year horizon.
Similar to the case of the American put option, we choose temperature parameter $\lambda\in\left\{10^{-2}, 10^{-4}, 10^{-6}\right\}$ and learning rate $\alpha\in\left\{0.1, 0.05, 0.01, 0.005\right\}$, and keep the other training parameters and structures of neural networks the same as in \cite{dai2024learningcost}, with 15000 training steps for both algorithms. As noted, in this example the stock price samples are simulated under the physical measure. Our experiments show that
among the 12 combinations of $\lambda$ and $\alpha$, \cite{dai2024learningcost}'s algorithm achieves higher average log return in the last training step than our algorithm for 2 times, while we achieve better results for 10 times.
For comparing the learning speeds, note that the omniscient level of average log return in the 20-year horizon under the testing dataset is around 3.00; so we choose the cases of $(\lambda,\alpha)$ in which the two algorithms achieve average log return values of greater than 2.95.
We find that $(\lambda,\alpha)=(10^{-6},0.05)$ and $(\lambda,\alpha)=(10^{-2},0.01)$ result in the fastest convergence of average log return for \cite{dai2024learningcost}'s algorithm and ours, respectively.
Figure \ref{fig:compare_costs} plots the corresponding learning curves of the initial free boundaries and average log returns.
Figure \ref{fig:compare_free_boundary} shows that the convergence of the sell boundary is almost identical for the two algorithms, while the convergence of the buy boundary is faster under our algorithm.
Figure \ref{fig:compare_avg_log_return} indicates that the convergence of average log return is around 5000 steps faster under our algorithm.
It is interesting to note that the trivial buy-and-hold strategy leads to an average log return value of 2.5; so both the RL algorithms indeed help the agent ``beat the market".


\begin{figure}[htb]
\centering
\subfigure[Free boundaries at the initial time]
{
\includegraphics[scale=0.45]{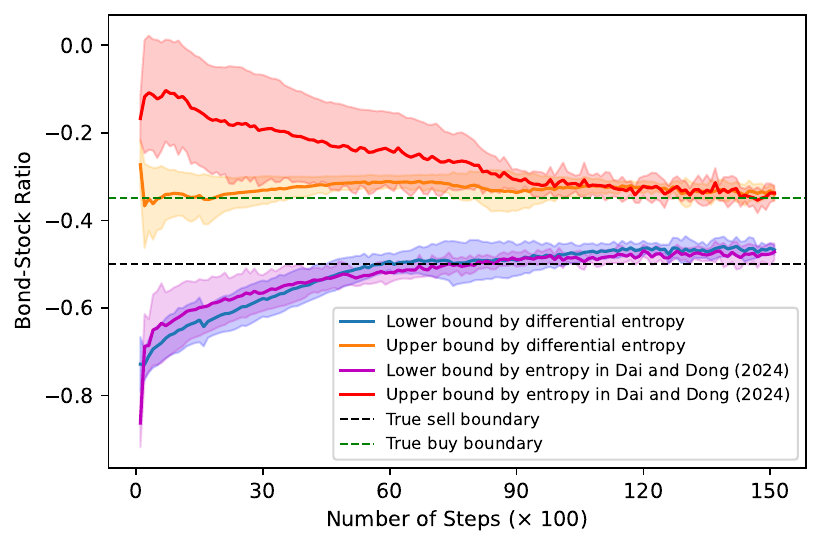}\label{fig:compare_free_boundary}
}
\subfigure[Log return]
{
\includegraphics[scale=0.45]{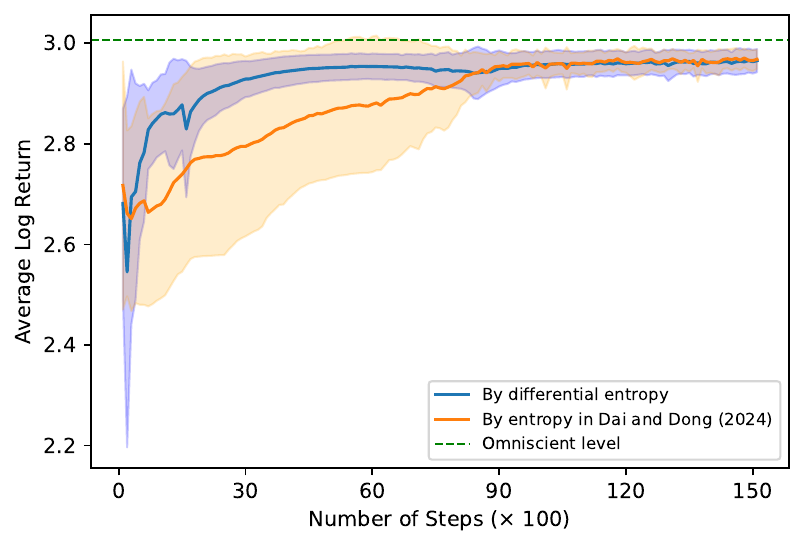}\label{fig:compare_avg_log_return}
}
\caption{Learning curves of free boundaries and average log return from our algorithm and \cite{dai2024learningcost}'s.
We repeat the training for 10 times to calculate the standard deviations of the learned free boundaries and average log return, which are represented as the shaded areas. The width of each shaded area is twice the corresponding standard deviation.}
\label{fig:compare_costs}
\end{figure}

\blue{One reason for the aforementioned difference in performance may be that compared to the differential entropy in our algorithm, the scale of the entropy term in \cite{dong2023randomized} is larger, causing greater biases in learning the optimal stopping value.
Specifically, the unnormalized negentropy $R(\tilde{\pi})\rightarrow -\infty$ as $\tilde{\pi}\rightarrow \infty$, whereas the differential entropy $-\BH(\pi)\in[0,\log2]$. The unboundedness of $R(\tilde{\pi})$ results in a convergence rate of $\frac{(CTe^{M/\lambda})^n}{n!}$ for policy iteration (see Theorem 3.4 in \cite{dong2023randomized}), where the proof implies that $C$ is a constant typically set to $2$,	
and $M$ depends on the temperature parameter $\lambda$. 	By contrast, Theorem \ref{convergence of policy improvement} shows a convergence rate of $\frac{(2KT)^n}{n!}$ for our algorithm.
To achieve sufficient accuracy, the penalty factor $K$ in our algorithm is typically chosen as $50$, while the temperature parameter in Dong’s algorithm must be sufficiently small, e.g., $\lambda\leq10^{-3}$.
The dependence of $M$ on $\lambda$ in \cite{dong2023randomized} is complex, as the unnormalized negentropy $R(\tilde{\pi})$ can take both positive and negative values. Nevertheless, based on the early-stage learning results shown in Figures \ref{fig:compare_learn_curve} and \ref{fig:compare_avg_log_return}, as well as the theoretical values---approximately 5.32 and 3.00, respectively, for the two problems---we have $M>1$.
Consequently, for small $n$, $\frac{(CTe^{M/\lambda})^n}{n!} \gg \frac{(2KT)^n}{n!}$, which partially explains the faster learning exhibited by our algorithm in the early training steps. As $n$ increases, the factorial in the denominator dominates, and both algorithms eventually converge.

In addition, in \cite{dong2023randomized}'s algorithm, the use of the unnormalized negentropy yields an optimal intensity $\tilde{\pi}=\exp\left(-\frac{V_{\theta}-g}{\lambda}\right)$, which takes values in $[0,\infty)$.
However, for a given time step size $\Delta t$, $\tilde{\pi}$ must be truncated at $\frac{1}{\Delta t}$ if $\exp\left(-\frac{V_{\theta}-g}{\lambda}\right)>\frac{1}{\Delta t}$, since $\tilde{\pi}\Delta t$ is interpreted as the probability of stopping in his formulation. This situation arises when the $\theta$-parameterized neural network generates a negative $V_{\theta}-g$ and the temperature parameter is small, such as  $\lambda\leq10^{-3}$ in the numerical examples. 	
As a result, the loss function minimized by stochastic gradient descent is not the exact TD error in his formulation (see equation (4.4) in \cite{dong2023randomized}), introducing a bias into the learning process. By contrast, our algorithm employs differential entropy resulting in an optimal probability $\pi\in[0,1]$, and the loss function minimized by stochastic gradient descent corresponds exactly to the TD error in our formulation. This may be another reason why our algorithm exhibits faster convergence.
}

\blue{
\section{Simulation studies: High-dimensional examples}\label{Sec:Simulation in high-dim}

Other than the model-free approach that mitigates the model estimation errors, another significant advantage of RL is its ability to scale to high-dimensional problems. We demonstrate this with two examples in this section.

\subsection{American geometric average put-type option}\label{SubSec:American geometric average put}

The following example is adapted from Section 4.3.2.1 in \cite{becker2021solving} and concerns the pricing of an American geometric average {put-type} option on up to 200 distinguishable stocks in a Black–Scholes market.

As in \cite{becker2021solving}, we set the dimension (the number of stocks)
$d\in\{40,80,120, \allowbreak 160,200\}$,
the initial stock price {$x_0=100^{1/\sqrt{d}}$ for all the stocks}, the time horizon $T=1$, the strike price $\Gamma=95$, the risk-free rate $\rho=0.6$, the volatility 
$\sigma_i=\min \{ 0.04 [(i - 1) \bmod 40], \, 1.6 - 0.04 [(i - 1) \bmod 40]\}$
and  dividend yield $\delta_i=\rho-\frac{\sum_{i=1}^d\sigma_i^2}{d^2}(i-\frac{1}{2})-\frac{1}{5\sqrt{d}}$ for each  $i\in\left\{1,2,\ldots,d\right\}$.
{Denote by $X_t^{(k)}$ the price of the $k$-th stock at time $t$.}
The optimal stopping problem is to
\begin{equation*}
\max_{\tau \in [0, T]} \, \mathbb{E} \left[ e^{-\rho \tau} \left(\Gamma - \left[ \prod_{k=1}^d |X_\tau^{(k)}|^{1/\sqrt{d}} \right]\right)^+ \right].
\end{equation*}

In our simulation studies, we set the penalty factor $K$ to be 100 (which corresponds to the time mesh grid $\Delta t=1/100$ {as discussed in Section \ref{Subsec:penalty_factor}}), the temperature parameter $\lambda=1$, the training batch size to be $2^{10}$, and the testing batch size to be $2^{17}$. The initial learning rate is set to be 0.05 and is decreased by a factor 1/10 every 300 training steps.
We train 1000 steps under the offline ML Algorithm \ref{algorithm_ML}.
In Table \ref{Table: American geometric average put-type option}, we show the learning results for different numbers $d$ of underlying stocks.
We test the learned price every 10 training steps; so we have 100 observations in total for each $d$. The final observation is listed in the second column of Table \ref{Table: American geometric average put-type option}.

It is worth noting that pricing this high-dimensional American geometric average put-type option is equivalent to pricing a one-dimensional American put option that is independent of the dimension.
\footnote{\blue{Specifically, for the equivalent one-dimensional American put-type option problem, the stock price process
in a Black--Scholes market setting
has an initial price of 100, a risk-free rate of $\rho$, a dividend yield of $\rho-\frac{\sum_{i=1}^d\sigma_i^2}{2d}-\frac{1}{5}$, and a volatility of $\frac{1}{\sqrt{d}}\sqrt{\sum_{i=1}^d\sigma_i^2}$ under the particular choices of $\sigma_i$ and $\delta_i$ in the example.  Notably, $\frac{\sum_{i=1}^d\sigma_i^2}{d}$ remains constant when the dimension is a multiple of 40; see Section 4.3.2.1 and Proposition 4.3 in \cite{becker2021solving}.}}
Hence we can easily obtain the corresponding omniscient values: The  theoretical price of its European counterpart is 3.565, and the benchmark value for the price of this one-dimensional American put option calculated using the binomial tree method on Smirnov’s website (\url{https://www.math.columbia.edu/~smirnov/options13.html}) with 20000 nodes is 6.545.
It turns out that our RL algorithm achieves high early exercise premium and is effective for pricing the high-dimensional option with comparable high accuracy compared to the benchmarks in \cite{becker2021solving}.
Moreover, our algorithm is scalable, as Table \ref{Table: American geometric average put-type option} shows a sublinear growth in runtime with respect to the dimension.


\begin{table}
\centering
\small
\begin{tabular}{cccc}
\toprule
Dimension $d$ & \parbox[t]{3.5cm}{\centering Last learned price \\by our method }  & In \cite{becker2021solving} & \parbox[t]{2cm}{\centering Runtime in sec.} \\
\midrule
40  & 6.512  & 6.512  & 1838 \\
80  & 6.523  & 6.509  & 2956 \\
120 & 6.504  & 6.507  & 4094 \\
160 & 6.507  & 6.504  & 6042 \\
200 & 6.493  & 6.501  & 7394 \\
\bottomrule
\end{tabular}
\caption{Comparison of learned prices and runtime for different values of dimension $d$. The initial stock price $x_0=100^{1/\sqrt{d}}$, the time horizon $T=1$, the strike price $\Gamma=95$, the risk-free rate $\rho=0.6$, the volatility $\sigma_i=\min \{ 0.04 [(i - 1) \bmod 40], \, 1.6 - 0.04 [(i - 1) \bmod 40]$ and dividend yield $\delta_i=\rho-\frac{\sum_{i=1}^d\sigma_i^2}{d^2}(i-\frac{1}{2})-\frac{1}{5\sqrt{d}}$ for each $i\in\left\{1,2,\ldots,d\right\}$. The payoff function when exercising the option is $g(x)=\left(\Gamma - \left[ \prod_{k=1}^d |x^{(k)}|^{1/\sqrt{d}} \right]\right)^+$.
The benchmark value for the price of this option is 6.545.}
\label{Table: American geometric average put-type option}
\end{table}

In the above experiment, the risk-free rate $\rho=0.6$ taken from \cite{becker2021solving} is unrealistically large. We now select a more reasonable value and adjust other parameters so that the equivalent one-dimensional American put option remains independent of {$d\in\left\{40,80,120,160,200\right\}$}. Specifically, we set the risk-free rate $\rho=0.06$, the volatility $\sigma_i=\frac{1}{\sqrt{10}}\min \left\{ 0.04 [(i - 1) \bmod 40], \, 1.6 - 0.04 [(i - 1) \bmod 40]\right\}$ and dividend yield $\delta_i=\rho-\frac{\sum_{i=1}^d\sigma_i^2}{d^2}(i-\frac{1}{2})-\frac{1}{50\sqrt{d}}$ for each $i\in\left\{1,2,\ldots,d\right\}$.
With these modifications, the equivalent one-dimensional American put option has a risk-free rate of 0.06, a dividend yield of 0.02932, and a volatility of 0.14615. We set the strike price to $\Gamma=110$. The theoretical price of the corresponding European option is 9.649 according to the Black-Scholes formula, and the benchmark value for the one-dimensional American put option, calculated using the binomial tree method on Smirnov’s website
with 20000 nodes, is 10.816. Table \ref{Table: American geometric average put-type option with risk-free rate 0.06} presents the learning results obtained by our ML algorithm for different dimensions, which again demonstrate the high accuracy and scalability of the algorithm.

\begin{table}
\centering
\small
\begin{tabular}{ccc}
\toprule
Dimension $d$ & \parbox[t]{3.5cm}{\centering Last learned price\\by our method }  & \parbox[t]{2cm}{\centering Runtime in sec.}\\
\midrule
40  & 10.828   & 1916 \\
80  & 10.803   & 2948\\
120 & 10.813   & 4141 \\
160 & 10.809   & 5665 \\
200 & 10.795   & 7732 \\
\bottomrule
\end{tabular}
\caption{Comparison of learned prices and runtime for different values of dimension $d$. The initial stock price $x_0=100^{1/\sqrt{d}}$, the time horizon $T=1$, the strike price $\Gamma=110$, the risk-free rate $\rho=0.06$, the volatility $\sigma_i=\frac{1}{\sqrt{10}}\min \{ 0.04 [(i - 1) \bmod 40], \, 1.6 - 0.04 [(i - 1) \bmod 40]$ and dividend yield $\delta_i=\rho-\frac{\sum_{i=1}^d\sigma_i^2}{d^2}(i-\frac{1}{2})-\frac{1}{50\sqrt{d}}$ for each $i\in\left\{1,2,\ldots,d\right\}$. The payoff function when exercising the option is $g(x)=\left(\Gamma - \left[ \prod_{k=1}^d |x^{(k)}|^{1/\sqrt{d}} \right]\right)^+$.
The benchmark value for the price of this option is 10.816.}
\label{Table: American geometric average put-type option with risk-free rate 0.06}
\end{table}

\subsection{Optimally stopping a fractional Brownian motion}\label{SubSec:fractional Brownian motion}
This example is adapted from Section 4.3 in \cite{becker2019deep} and Section 4.2.3 in \cite{dong2023randomized}, and is an optimal stopping problem for a fractional Brownian motion with Hurst parameter $H\in(0,1]$. By augmenting the state variable to include the entire history of the process, the problem becomes Markovian, and numerically a high-dimensional optimal stopping problem.

Specifically, a fractional Brownian motion with Hurst parameter $H \in (0, 1]$ is a continuous, centered Gaussian process $(W_t^H)_{t \geq 0}$ with covariance structure
\begin{equation*}
\mathbb{E}[W_t^H W_s^H] = \frac{1}{2} \left( t^{2H} + s^{2H} - |t - s|^{2H} \right),~~t,s\geq 0.
\end{equation*}
When $H=1/2$, $W^H$ is a standard Brownian motion, and  when $H\neq 1/2$, $W^H$ is neither a martingale nor a Markov process.
The objective is to
\begin{align*}
\max_{0\leq\tau\leq 1}  \mathbb{E}[W_{\tau}^H]
\end{align*}
over all $W^H$-stopping times $\tau\in[0,1]$.
This problem is non-Markovian and therefore falls outside the primary scope of this paper. However, we can approximate the optimal stopping value by discretizing the process. When combined with state space augmentation, this discretization transforms the problem into a Markovian one, thereby enabling the application of the martingale-based approach developed in this work.

To do this, denote $t_n=n/100$, for $n=0,1,\ldots,100$. For details on simulating trajectories of fractional Brownian motion, see \cite{becker2019deep}. Define the following 101-dimensional state process:
\begin{align*}
X_0 & = \left(W_{t_0}^H, W_{t_0}^H,W_{t_0}^H,\ldots,W_{t_0}^H\right),\quad\quad\quad\quad\
X_1  = \left(W_{t_1}^H,W_{t_0}^H, W_{t_0}^H,\ldots,W_{t_0}^H\right),\\
X_2 & = \left(W_{t_2}^H, W_{t_1}^H, W_{t_0}^H,\ldots,W_{t_0}^H\right),\quad \ldots, \quad
X_{100}  = \left(W_{t_{100}}^H, W_{t_{99}}^H, W_{t_{98}}^H,\ldots,W_{t_{0}}^H\right),
\end{align*}
and consider the following optimal stopping problem:
\begin{align*}
\max_{\tau\in\mathcal{T}}  \mathbb{E}[g(X_{\tau})],
\end{align*}
where $g(x)=x^{(1)}$ for a multi-dimensional vector $x$ with first component $x^{(1)}$, and $\mathcal{T}$ denotes the set of all $X$-stopping times.
The multi-dimensional state process  $X_{0}$, $X_{1}$, $\cdots $, $X_{100}$  is now Markovian, and the martingale-based approach applies. 


In our simulation, we set the penalty factor $K$ to be 100 (corresponding to a time mesh grid $\Delta t=1/100$ {as specified in Section \ref{Subsec:penalty_factor}}), the temperature parameter $\lambda$ to be 0.1, the training batch size to be $2^{10}$, the testing batch size to be $2^{15}$, and the initial learning rate to 0.01 which is halved every 200 training steps. We train for 3000 steps using Algorithm \ref{algorithm_ML}.

Table \ref{Table: fractional Brownian motion} presents the learning results for different Hurst parameter values. The learned value is evaluated every 10 training steps, resulting in 300 observations for each case. The final observation is reported in the second column of Table \ref{Table: fractional Brownian motion}.
Note that in this example, if the agent does not stop until the terminal time, the stopping value is 0 in theory. Thus, our RL algorithm can achieve a high {stopping value}. Our learning results are comparable to benchmarks in the literature \citep{becker2019deep, dong2023randomized}, demonstrating high accuracy in learning the stopping decisions in a non-Markovian (or equivalently a high-dimensional) setting.

\begin{table}
\centering
\small
\begin{tabular}{ccccc}
\toprule
$H$ & \parbox[t]{3.5cm}{\centering Last learned value\\by our method} & In \cite{dong2023randomized} & In \cite{becker2019deep} & \parbox[t]{2cm}{\centering Runtime in sec.} \\
\midrule
0.01 & 1.496  & 1.513 & 1.519 & 3013 \\
0.05 & 1.270  & 1.283 & 1.293 & 3008 \\
0.10 & 1.027  & 1.043 & 1.049 & 3039 \\
0.15 & 0.821  & 0.825 & 0.839 & 3476 \\
0.20 & 0.641  & 0.651 & 0.658 & 3539 \\
0.25 & 0.490  & 0.494 & 0.503 & 3519 \\
0.30 & 0.360  & 0.364 & 0.370 & 3525 \\
0.35 & 0.246  & 0.243 & 0.255 & 3532 \\
0.40 & 0.148  & 0.144 & 0.156 & 3519 \\
0.45 & 0.059  & 0.061 & 0.071 & 3519 \\
0.50 & 0.002  & 0.000 & 0.002 & 3526 \\
0.55 & 0.056  & 0.052 & 0.061 & 3528 \\
0.60 & 0.110  & 0.108 & 0.117 & 3516 \\
0.65 & 0.158  & 0.155 & 0.164 & 3528 \\
0.70 & 0.201  & 0.192 & 0.207 & 3352\\
0.75 & 0.240  & 0.233 & 0.244 & 3080\\
0.80 & 0.274  & 0.261 & 0.277 & 3009 \\
0.85 & 0.306  & 0.299 & 0.308 & 2454 \\
0.90 & 0.337  & 0.331 & 0.337 & 2536 \\
0.95 & 0.366  & 0.354 & 0.366 & 2370\\
\bottomrule
\end{tabular}
\caption{Comparison of stopping values and runtime for different values of Hurst parameter \( H \) in the example of optimally stopping a fractional Brownian motion. The horizon length $T=1$.}
\label{Table: fractional Brownian motion}
\end{table}

}

\section{Conclusions}\label{Sec:Conclusion}

In this paper, we study optimal stopping for diffusion processes in the exploratory RL framework of \cite{wang2020reinforcement}. Different from \cite{dong2023randomized} and \cite{dai2024learningcost} which directly randomize stopping times for exploration through mixed strategies characterized by Cox processes, we turn the problem into stochastic control by approximating the variational inequality with the penalized PDE. This enables us to apply all the available results and methods developed recently in RL for controlled diffusion processes \citep{jia2022policy,jia2022policygradient,jia2023q}.
Our simulation study shows that the resulting RL algorithms achieve comparable, and sometimes better, performance in terms of accuracy and speed than those in \cite{dong2023randomized} and \cite{dai2024learningcost}. \blue{More importantly, we show experimentally that the proposed algorithms also work well for high-dimensional problems.

There is a notable limitation of the current approach especially to high-dimensional problems. In our implementation, we use neural networks to approximate value functions, which enables us to overcome the curse of dimensionality. However, a challenge may arise if the value functions—such as discontinuous ones—cannot be well approximated by neural networks. Overcoming the difficulty is an interesting yet important question. }

We take option pricing as an application of the general theory in our experiments. There is a caveat in this: the expectation in the pricing formula is with respect to the risk-neutral probability under which data are {\it not observable}. While we can use simulation to test the learning efficiency of an RL algorithm, for real applications we need to take a formulation for option pricing that is under the physical measure due to the data-driven nature of RL. One example of such a formulation is the mean--variance hedging.

There are many other interesting problems to investigate along the lines of this paper, including a more substantial convergence and regret analysis, choice of the temperature parameter, and more complex underlying processes to stop.

\appendix

\section{Proofs of statements}

\subsection{Proof of Lemma \ref{lemma_K}}\label{Appendix: proof_penalty}


Suppose there is a positive sequence $\{K_n\}_{n\geq 1}$ that increases to infinity, and $v^{(n)}$ solves \eqref{hjb2} with $K=K_n$.
Denote  $a_n=v_t^{(n)}+\opla v^{(n)},\ b_n=g-v^{(n)}$,
$a_{\infty}=v_t+\opla v$, and $ b_{\infty}=g-v$,
where $v$ is the limit of $v^{(n)}$.
From \eqref{hjb2} we have $a_n=- K_n(b_n)^{+}\leq 0$.
Because $\left(a_n, b_n\right)$ converges to $\left(a_{\infty}, b_{\infty}\right)$ as $n\rightarrow \infty$, we get $a_{\infty}\leq 0$. On the other hand, on any compact region, $\{a_n\}_{n\geq 1}$ is locally bounded; so $(b_n)^+=-\frac{a_n}{K_n} \rightarrow 0, \text{as $n\rightarrow\infty$}$,
yielding $b_{\infty}\leq 0$.

If $b_{\infty}<0$, then there exists a large number $N$ such that $b_n<0$ for all $n>N$,
leading to $a_n=0$ for all $n>N$.
Thus, $a_{\infty}=0$ when $b_{\infty}<0$.

Hence, we have $a_{\infty}\leq 0$, $b_{\infty}\leq 0$ and $a_{\infty}b_{\infty}=0$,
which implies that $v$ solves \eqref{hjbp1}.

\subsection{Proof of Theorem \ref{vf}}\label{proof_HJB}
For notational simplicity, without loss of generality we only consider the case when $X$ and $W$ are both one-dimensional.

Write $\overline{R}_{t}=e^{-\int_{s}^{t}f(r, X_{r}, u_{r})\dr}, \ t\geq s$. Then $\dd \overline{R}_{t}=-\overline{R}_{t}f(t, X_{t}, u_{t})\dt, \ \overline{R}_{s}=1$.
It\^{o}'s lemma gives
\begin{equation*}
	\begin{aligned}
		&\quad\;\dd\; (\overline{R}_{t}v(t,X_{t})) \\
		&=\overline{R}_{t}\dd v(t,X_{t})+v(t,X_{t}) \dd \overline{R}_{t}\\
		&=\overline{R}_{t}\big[\left(v_{t}(t,X_t)+\opla_{u} v(t,X_t)\right)\dt+v_x(t,X_t)\sigma(t,X_t, u_{t})\dd W_{t}\big]-v(t,X_{t}) \overline{R}_{t} f(t, X_{t}, u_{t})\dt\\
		&=\overline{R}_{t}\big[\left(v_{t}(t,X_t)+\opla_{u} v(t,X_t)-v(t,X_t)f(t, X_{t}, u_{t})\right)\dt+v_x(t,X_t)\sigma(t,X_t, u_{t})\dd W_{t}\big]\\
		&\leq -\overline{R}_{t} G(t, X_{t}, u_{t}) \dt+\overline{R}_{t} v_x(t,X_t)\sigma(t,X_t, u_{t})\dd W_{t},
	\end{aligned}
\end{equation*}
where the last inequality is due to \eqref{hjbpp1} and $\overline{R}_{t}>0$.
Integrating both sides from $s$ to $T$ and then taking conditional expectation yield
\begin{equation*}
	\BE{\overline{R}_{T}v(T,X_{T})-\overline{R}_{s}v(s,X_{s})\;\big|\; X_s=x}
	\leq -\BE{\int_s^{T}\overline{R}_{t} G(t, X_{t}, u_{t})\dt\;\bigg|\; X_s=x }.
\end{equation*}
As $\overline{R}_{s}=1$, the above leads to
\begin{equation*}
	\begin{aligned}
		v(s, x)
		&\geq \BE{\int_s^{T}\overline{R}_{t} G(t, X_{t}, u_{t})\dt+\overline{R}_{T}v(T,X_{T})\;\bigg|\; X_s=x}\\
		&=\BE{\int_s^{T} e^{-\int_{s}^{t}f(r, X_{r}, u_{r})\dr}G(t, X_t, u_{t})\dt
			+e^{-\int_{s}^{T}f(r, X_{r}, u_{r})\dr}H(X_T)\;\bigg|\; X_s=x}.
	\end{aligned}
\end{equation*}
This means that $v$ is an upper bound of the optimal value function of problem \eqref{pp1}.
All the inequalities above become equalities if we take $u\in U$ that solves the optimization problem in \eqref{hjbpp1}.
Therefore, we conclude $v$ is the optimal value function of \eqref{pp1}.

\subsection{Proof of Theorem \ref{policy improvement}}
By the Feynman--Kac formula, for any $(t,x)$,
\begin{equation*}
	\partial_t J^{\pi}(t,x)+\opla J^{\pi}(t,x)+H\left(t,x,\pi,J^{\pi}\right)=0,
\end{equation*}
where
\begin{equation*}
	H\left(t,x,\pi,J^{\pi}\right) :=K \pi(t,x)\left(g(t,x)-J^{\pi}(t,x)\right) - \lambda \BH(\pi(t,x)).
\end{equation*}
Similarly, for any $(t,x)$,
\begin{equation*}
	\partial_t J^{\tilde{\pi}}(t,x)+\opla J^{\tilde{\pi}}(t,x)+H\left(t,x,\tilde{\pi},J^{\tilde{\pi}}\right)=0,
\end{equation*}
where $\tilde{\pi}$ is defined in \eqref{tilde_pi}. It turns out that
\begin{equation*}
	\tilde{\pi}(t,x)=\argmax_{\pi\in[0, 1]} H\left(t,x,\pi,J^{\pi}\right).
\end{equation*}

Denote $\Delta(t,x):=J^{\tilde{\pi}}(t,x)-J^{\pi}(t,x)$. Then
\begin{equation*}
	\partial_t \Delta(t,x)+\opla \Delta(t,x)+H\left(t,x,\tilde{\pi},J^{\tilde{\pi}}\right) - H\left(t,x,\pi,J^{\pi}\right)=0.
\end{equation*}
Because
\begin{equation*}
	\begin{aligned}
		& H\left(t,x,\tilde{\pi},J^{\tilde{\pi}}\right) - H\left(t,x,\pi,J^{\pi}\right) \\
		=&\ H\left(t,x,\tilde{\pi},J^{\tilde{\pi}}\right) - H\left(t,x,\tilde{\pi},J^{\pi}\right)+
		H\left(t,x,\tilde{\pi},J^{\pi}\right) - H\left(t,x,\pi,J^{\pi}\right) \\
		=& - K \tilde{\pi}(t,x)\Delta(t,x)+H\left(t,x,\tilde{\pi},J^{\pi}\right) - H\left(t,x,\pi,J^{\pi}\right),
	\end{aligned}
\end{equation*}
we have
\begin{equation*}
	\partial_t \Delta(t,x)+\opla \Delta(t,x) - K \tilde{\pi}(t,x)\Delta (t,x)=H\left(t,x,\pi,J^{\pi}\right) - H\left(t,x,\tilde{\pi},J^{\pi}\right)\leq 0,
\end{equation*}
where the last inequality is from the definition of $\tilde{\pi}$.
Besides, $\Delta(T,x)=0$ at the terminal time $T$.
Therefore, 0 is a sub-solution of $-\partial_t u(t,x) - \opla u(t,x)+K \tilde{\pi}(t,x) u(t,x)+H\left(t,x,\pi,J^{\pi}\right) - H\left(t,x,\tilde{\pi},J^{\pi}\right)=0$.
By the maximum principle, we have $\Delta(t,x)\geq 0$, resulting in $J^{\tilde{\pi}}(t,x)\geq J^{\pi}(t,x)$.

\blue{
	\subsection{Proof of Theorem \ref{convergence of policy improvement}}\label{Appendix:proof of convergence of policy improvement}
	
	Recall that $v^{\boldsymbol{\Pi}}$ is the optimal value function of the exploratory control problem \eqref{p3} with the optimal strategy \eqref{equ update pi}.
	By the same procedure as in the proof of Theorem \ref{policy improvement}, we can show that $v^{\boldsymbol{\Pi}}\geq \cdots\geq J^{\pi^{n+1}}\geq J^{\pi^n}\geq\cdots\geq J^{\pi^0}$.
	Next, we show the uniform convergence of the sequence $\left\{J^{\pi^n}\right\}_{n=1}^{\infty}$. To proceed, write $\left\|g\right\|_{\infty}= C$ for some positive constant $C$ due to the assumption that $g$ is uniformly bounded.
	
	
	For any $n\geq 1$, denote
	\begin{align*}
		f^n(t):= &\;\sup_x \left\{v^{\boldsymbol{\Pi}}(t,x)-J^{\pi^{n}}(t,x)\right\}\geq 0,\\
		\phi^{n}(t) :=&\; 2K \int_t^T f^n(s)ds,\\
		W(t,x):=&\;v^{\boldsymbol{\Pi}}(t,x)-J^{\pi^{n+1}}(t,x)-\phi^{n}(t).
	\end{align*}
	Using
	$	\partial_t J^{\pi^{n+1}}+\opla J^{\pi^{n+1}} + H\left(t,x,\pi^{n+1},J^{\pi^{n+1}}\right) = 0$
	and
	$	\partial_t v^{\boldsymbol{\Pi}}+\opla v^{\boldsymbol{\Pi}} + H\left(t,x,\pi^*,v^{\boldsymbol{\Pi}}\right) = 0$,
	we have
	\begin{align*}
		\partial_t W+\opla W & = \partial_t v^{\boldsymbol{\Pi}}+\opla v^{\boldsymbol{\Pi}} - \left(\partial_t J^{\pi^{n+1}}+\opla J^{\pi^{n+1}}\right) - \partial_t \phi^{n}\\
		&= H\left(t,x,\pi^{n+1},J^{\pi^{n+1}}\right) - H\left(t,x,\pi^*,v^{\boldsymbol{\Pi}}\right) - \partial_t \phi^{n}\\
		&= H\left(t,x,\pi^{n+1},J^{\pi^{n+1}}\right) - H\left(t,x,\pi^{n+1},v^{\boldsymbol{\Pi}}\right) \\&\quad+ H\left(t,x,\pi^{n+1},v^{\boldsymbol{\Pi}}\right) - H\left(t,x,\pi^*,v^{\boldsymbol{\Pi}}\right) - \partial_t \phi^{n}.
	\end{align*}
	Noticing that
	\begin{align*}
		H\left(t,x,\pi,J^{\pi}\right) &=K \pi(t,x)\left(g(t,x)-J^{\pi}(t,x)\right) - \lambda \BH(\pi(t,x)),\\
		\BH(\pi^{n+1}(t,x)) &= \pi^{n+1}\log \pi^{n+1} + (1-\pi^{n+1}) \log\left(1-\pi^{n+1}\right)\\
		& = \lambda^{-1}K \left(J^{\pi^n}-g\right)\left(1-\pi^{n+1}\right)-\log\left(1+\exp\left(\lambda^{-1}K\left(J^{\pi^n}-g\right)\right)\right),
	\end{align*}
	and
	\begin{align*}
		\BH(\pi^{*}(t,x)) &= \pi^{*}\log \pi^{*} + (1-\pi^{*}) \log\left(1-\pi^{*}\right)\\
		& = \lambda^{-1}K \left(v^{\boldsymbol{\Pi}}-g\right)\left(1-\pi^{*}\right)-\log\left(1+\exp\left(\lambda^{-1}K\left(v^{\boldsymbol{\Pi}}-g\right)\right)\right),
	\end{align*}
	we obtain
	\begin{align*}
		H\left(t,x,\pi^{n+1},v^{\boldsymbol{\Pi}}\right) & = K\pi^{n+1} \left(g-v^{\boldsymbol{\Pi}}\right) - K\left(J^{\pi^n}-g\right)\left(1-\pi^{n+1}\right)\\
		&\quad+\lambda\log\left(1+\exp\left(\lambda^{-1}K\left(J^{\pi^n}-g\right)\right)\right)
	\end{align*}
	and
	\begin{align*}
		H\left(t,x,\pi^{*},v^{\boldsymbol{\Pi}}\right) & = K\pi^{*} \left(g-v^{\boldsymbol{\Pi}}\right) - K\left(v^{\boldsymbol{\Pi}}-g\right)\left(1-\pi^{*}\right)\\
		&\quad +\lambda\log\left(1+\exp\left(\lambda^{-1}K\left(v^{\boldsymbol{\Pi}}-g\right)\right)\right)\\
		& = K\left(g-v^{\boldsymbol{\Pi}}\right) +\lambda\log\left(1+\exp\left(\lambda^{-1}K\left(v^{\boldsymbol{\Pi}}-g\right)\right)\right).
	\end{align*}
	Straightforward calculations show that
	\begin{align*}
		H\left(t,x,\pi^{n+1},J^{\pi^{n+1}}\right) - H\left(t,x,\pi^{n+1},v^{\boldsymbol{\Pi}}\right) = K\pi^{n+1}\left(v^{\boldsymbol{\Pi}}-J^{\pi^{n+1}}\right) \geq 0
	\end{align*}
	and
	\begin{align*}
		&H\left(t,x,\pi^{n+1},v^{\boldsymbol{\Pi}}\right) - H\left(t,x,\pi^*,v^{\boldsymbol{\Pi}}\right) \\
		=&\ K \left(v^{\boldsymbol{\Pi}}-J^{\pi^{n}}\right)\frac{1}{1+\exp\left(-\lambda^{-1}K\left(J^{\pi^n}-g\right)\right)}
		- \lambda \log \left(\frac{1+\exp\left(\lambda^{-1}K\left(v^{\boldsymbol{\Pi}}-g\right)\right)}{1+\exp\left(\lambda^{-1}K\left(J^{\pi^n}-g\right)\right)}\right).
	\end{align*}
	It follows from the general inequality $\left|\log(1+\exp(x))-\log(1+\exp(y))\right|\leq\left|x-y\right|$ that
	\begin{align*}
		\left| H\left(t,x,\pi^{n+1},v^{\boldsymbol{\Pi}}\right) - H\left(t,x,\pi^*,v^{\boldsymbol{\Pi}}\right)\right| \leq 2K \left(v^{\boldsymbol{\Pi}}-J^{\pi^{n}}\right)\leq 2K f^n(t).
	\end{align*}
	Thus,
	\begin{align*}
		& H\left(t,x,\pi^{n+1},v^{\boldsymbol{\Pi}}\right) - H\left(t,x,\pi^*,v^{\boldsymbol{\Pi}}\right) - \partial_t \phi^{n} \\
		=&\ H\left(t,x,\pi^{n+1},v^{\boldsymbol{\Pi}}\right) - H\left(t,x,\pi^*,v^{\boldsymbol{\Pi}}\right) + 2K f^n(t) \geq 0.
	\end{align*}
	Consequently,
	\begin{align*}
		\partial_t W+\opla W  \geq 0,
	\end{align*}
	which together with $W(T,x)=0$ implies  $W\leq 0$ by the maximum principle. Thus, we conclude
	\begin{align*}
		v^{\boldsymbol{\Pi}}(t,x)-J^{\pi^{n+1}}(t,x)\leq \phi^{n}(t) = 2K \int_t^T f^n(s)ds.
	\end{align*}
	The definition of $f^{n}$ yields
	\begin{align*}
		f^{n+1}(t)\leq 2K \int_t^T f^n(s)ds.
	\end{align*}
	By induction, we have
	\begin{align*}
		f^{n+1}(t)\leq \left(2K\right)^n \frac{(T-t)^n}{n!}  \sup_t f^1(t)\leq
		\frac{(2KT)^n}{n!}  \sup_t f^1(t).
	\end{align*}
	
	Recall
	\begin{equation*}
		J^{\pi}(s,x)=\BE{\int_s^{T} e^{-K\int_{s}^{t}\pi_{v}\dv}\big[K g(t,X_t) \pi_{t}-\lam \BH(\pi_{t})\big]\dt+e^{-K\int_{s}^{T}\pi_{v}\dv}g(T,X_T) \;\bigg|\; X_s=x },
	\end{equation*}
	$\left\|g\right\|_{\infty}= C$,  $e^{-K\int_{s}^{t}\pi_{v}\dv}\leq 1$ for any $0\leq t\leq T$, and $0\leq \pi_t\leq 1$ and $$|\BH(\pi_{t})|\leq  \sup_{x\in(0,1)}|x\log x+(1-x)\log (1-x)|=\log 2.$$ We thus have
	\begin{equation*}
		|J^{\pi}(s,x)|\leq \BE{\int_s^{T}  \big[KC+\lam \log 2\big]\dt+C\;\bigg|\; X_s=x }\leq (KC+\lam\log 2)T+C
	\end{equation*}
	for any admissible $\pi$.
	It hence follows
	\begin{align*}
		\sup_t f^1(t) = \sup_{t,x} \left\{v^{\boldsymbol{\Pi}}(t,x)-J^{\pi^{1}}(t,x)\right\} & \leq  2((KC+\lam\log 2)T+C).
	\end{align*}
	
	Consequently,
	\begin{align*}
		\left\| v^{\boldsymbol{\Pi}} - J^{\pi^{n+1}} \right\|_{\infty}
		\leq   \frac{(2KT)^n}{n!}  2((KC+\lam\log 2)T+C),
	\end{align*}
	which converges to 0 as $n$ increases to infinity when $K$ and $\lambda$ are fixed. This concludes the proof.

	Finally, let us complement the proof by 
	numerically examining the theoretical convergence rate. 
	As shown in Figure \ref{fig:numerical_diagnostics}, the NN approximation for the stopping value of $(t_0,x_0)=( 0,40)$ converges to its theoretical value 5.333, which is calculated using the finite difference method, and the relative error decreases exponentially. In implementation, we use stochastic gradient descent to update the NN parameters at each training step, which prevents us from performing the policy evaluation in \eqref{Eq: PolicyEvaluateFeynmanKac} with a complete accuracy. Consequently, in the experiment we do not achieve the
	$n!$ rate of decay in relative error as predicted by Theorem \ref{convergence of policy improvement}. Nevertheless, exponential decay is sufficient to ensure rapid convergence.
	
	\begin{figure}[htb]
		\centering
		\subfigure
		{
			\includegraphics[scale=0.45]{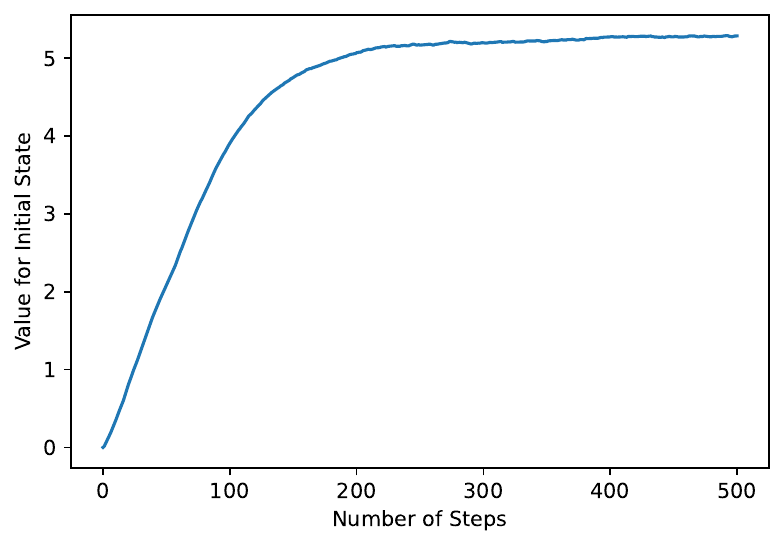}\label{fig:numerical_diagnostics_value_initial}
		}
		\subfigure
		{
			\includegraphics[scale=0.45]{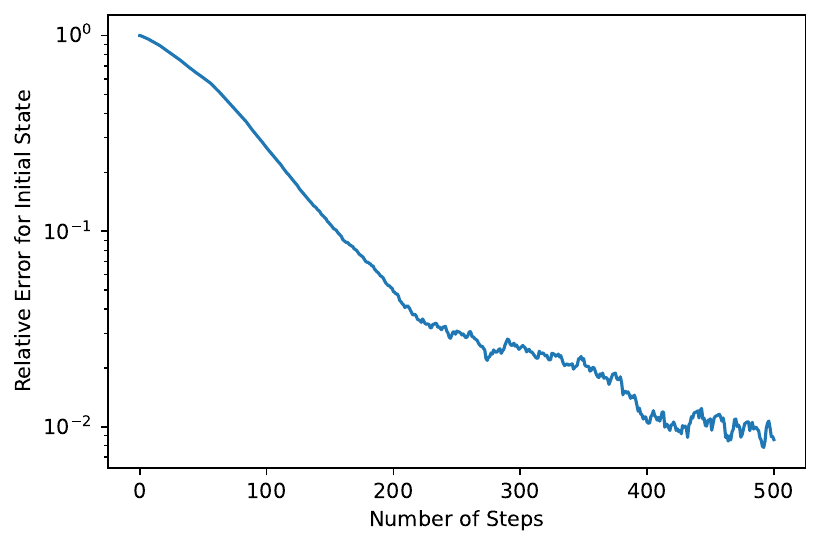}\label{fig:numerical_diagnostics_log_relative_error_initial}
		}
		\caption{	
			Numerical diagnostics for the convergence of the RL algorithms. We use the offline ML Algorithm \ref{algorithm_ML} to learn the stopping value of $(t_0,x_0)=( 0,40)$, approximated by an NN, for the American put example in Section \ref{Sec:Simulation} with $\lambda=0.01$ and $K=50$. The left panel displays the learning curve of the value and the right panel shows that of the relative error, both with respect to training steps. The y-axis of the right panel is plotted on a logarithmic scale.}
		\label{fig:numerical_diagnostics}
	\end{figure}
}

\section{Learning early exercise premium}\label{learn early exercise premium}
In this section, we discuss how to learn the early exercise premium of an American put option.

Recall we have an explicit pricing formula for a European put option:
\begin{equation}
	\begin{cases}\label{European_put_value}
		V_E(t, x)=\Gamma e^{-\rho(T-t)}N\left(-d_{-}(t,x)\right) - xN\left(-d_{+}(t,x)\right), \ t\in[0,T),\\
		V_E(T,x)=g(x),
	\end{cases}	
\end{equation}
where $d_{\pm}(t,x)=\frac{1}{\sigma\sqrt{T-t}}\left[\log\left(\frac{x}{\Gamma}\right)+\left(\rho\pm\frac{1}{2}\sigma^2\right)\left(T-t\right)\right]$,
and $N(\cdot)$ is the cumulative standard normal distribution function $N(y)=\frac{1}{\sqrt{2\pi}}\int_{-\infty}^y e^{-\frac{z^2}{2}} \dz=\frac{1}{\sqrt{2\pi}}\int_{-y}^{\infty} e^{-\frac{z^2}{2}} \dz$.

The early exercise premium of the American put option is
\begin{equation}
	\tilde{v}(t,x):=v(t,x)-V_E(t,x),
\end{equation}
where $v(t,x)$ is the option value at the time--state pair $(t,x)$.
The associated payoff function is
\begin{equation}
	\tilde{g}(t,x):=g(x)-V_E(t,x).
\end{equation}
By \eqref{hjbp1} and the fact that $V_E$ satisfies $\frac{\partial}{\partial t}V_E+\opla V_E=0$ with terminal condition $V_E(T,x)=g(x)$,
we have
\begin{equation}
	\begin{cases}
		\max\{\tilde{v}_{t}+\opla \tilde{v}, \ \tilde{g}-\tilde{v}\}=0,\ (t,x)\in[0,T)\times\R, \\
		\tilde{v}(T,x)=0.
	\end{cases}
\end{equation}
Thus, the terminal condition is smooth now, and the related PDE by penalty approximation is
\begin{equation*}
	\begin{cases}
		\tilde{v}_{t}+\max_{u\in\{0,1\}} \left\{\opla \tilde{v}- Ku\tilde{v}+K\tilde{g}(t,x)u\right\}=0,
		\ (t,x)\in[0,T)\times\R,\\
		\tilde{v}(T,x)=0.
	\end{cases}
\end{equation*}
Consequently, by Theorem \ref{vf} the corresponding control problem is
\begin{equation}
	\max_{u\in \{0,1\} }\; \BE{\int_s^{T} Ke^{-K\int_{s}^{t}u_{r}\dr}\tilde{g}(t,X_t)u_{t}\dt \;\bigg|\; X_s=x},
\end{equation}
and the entropy-regularized exploratory problem is
\begin{equation}\label{v_tilde randomized}
	\max_{\pi\in [0,1]} \; \BE{\int_s^{T} e^{-K\int_{s}^{t}\pi_{v}\dv}\big[K \tilde{g}(t,X_t) \pi_{t}-\lam \BH(\pi_{t})\big]\dt \;\bigg|\; X_s=x }.
\end{equation}
While we could use the algorithms in Section \ref{Section: RL algorithms} to learn $\tilde{v}$, $\tilde{g}$ is not fully known to us since $\sigma$ (which appears in $V_E$) is an unknown model parameter.

To deal with the unknown payoff function $\tilde{g}$,
recall
\begin{equation}\label{payoff_g_tilde}
	\tilde{g}(t,x)=g(x) - V_E(t,x)=\left(\Gamma - x\right)^+- V_E(t,x).
\end{equation}
Hence, in the formula of $V_E$, namely \eqref{European_put_value}, we replace $\sigma$ by $\phi$, which is a parameter to learn with a properly chosen initial value. Denote by $V_E^{\phi}$ and $\tilde{g}^{\phi}$ the parameterized European option value function and payoff function, respectively. 
Note that
\begin{equation}\label{V_E expectation}
	V_E(t,x)=\mathbb{E}\left[e^{-\rho(T-t)}g(X_T) \;\bigg|\; X_t=x\right]=\mathbb{E}\left[e^{-\rho(T-t)}\left(\Gamma-X_T\right)^+\;\bigg|\; X_t=x\right].
\end{equation}
This is a policy evaluation problem. Thanks to the martingale approach developed in \cite{jia2022policy}, we can learn/update $\phi$ by minimizing the martingale loss w.r.t. $\phi$:
\begin{equation}
	\begin{aligned}\label{martingale_loss_payoff}
		\quad\quad\text{ML}\left(\phi\right)
		&\approx \frac{1}{2}\mathbb{E}\left[\sum_{l=0}^{L-1}\left( e^{-\rho T}V_E\left(T,X_T\right)- e^{-\rho t} V_E^{\phi}\left(t_l,X_l\right) \right)^2\Delta t\right]\\
		&=\frac{1}{2}\mathbb{E}\left[\sum_{l=0}^{L-1}\left (e^{-\rho T}\left(\Gamma-X_T\right)^+- e^{-\rho t} V_E^{\phi}\left(t_l,X_l\right)\right)^2\Delta t\right].
	\end{aligned}
\end{equation}
In each episode, we simulate a batch of stock price trajectories and then use the stochastic gradient descent method to minimize the above loss function to update $\phi$.
Figure \ref{fig:learn payoff} shows the efficacy of this policy evaluation algorithm, with an initial guess of the unknown $\sigma$ to be 0.8, twice the true value.

\begin{figure}[htb]
	\centering
	\includegraphics[scale=0.45]{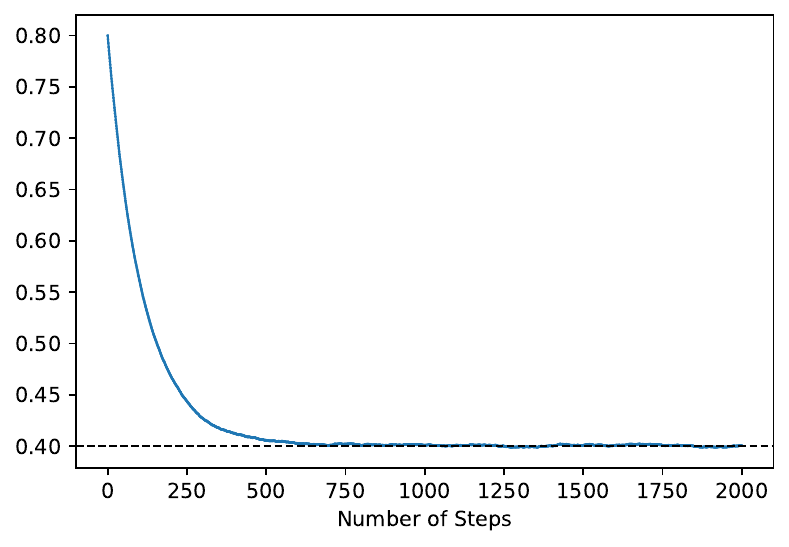}
	\caption{Learning curve of $\phi$, which is used to parameterize the payoff function \eqref{payoff_g_tilde} by minimizing the martingale loss \eqref{martingale_loss_payoff}. The initialization of $\phi$ is 0.8, and the true value of $\sigma$ is 0.4. The value of $\phi$ after 2000 steps is 0.4006.}
	\label{fig:learn payoff}
\end{figure}

We emphasize that in this particular example of European option, due to the availability of the explicit formula of the option price, the critic parameter $\phi$ {\it happens} to be a model parameter $\sigma$.
In general, for more complicated problems with little structure, the actor/critic parameters may be very different from the model primitives and may themselves have no physical/practical meanings at all. Estimating model parameters is {\it not} a goal of RL in our framework. 


\section{Additional results for sensitivity tests}\label{Appendix: K=50}
In Figure \ref{fig: learn_curve_ML_TD0_by_VE_as_lam_varies_fix_K_49}, we show a sensitivity test when the penalty factor approaches its maximum possible value, i.e., $K=50$, under our time discretization setting. In this case, the two ways of calculating the option price have almost the same results, which is consistent with Lemma \ref{lemma_K}. The trade-off between learning accuracy and speed observed previously still holds.

\begin{figure}[htb]
	\centering
	\subfigure[$\mathcal{P}_{\mathrm{stopping}}$ by ML]
	{\includegraphics[scale=0.38]{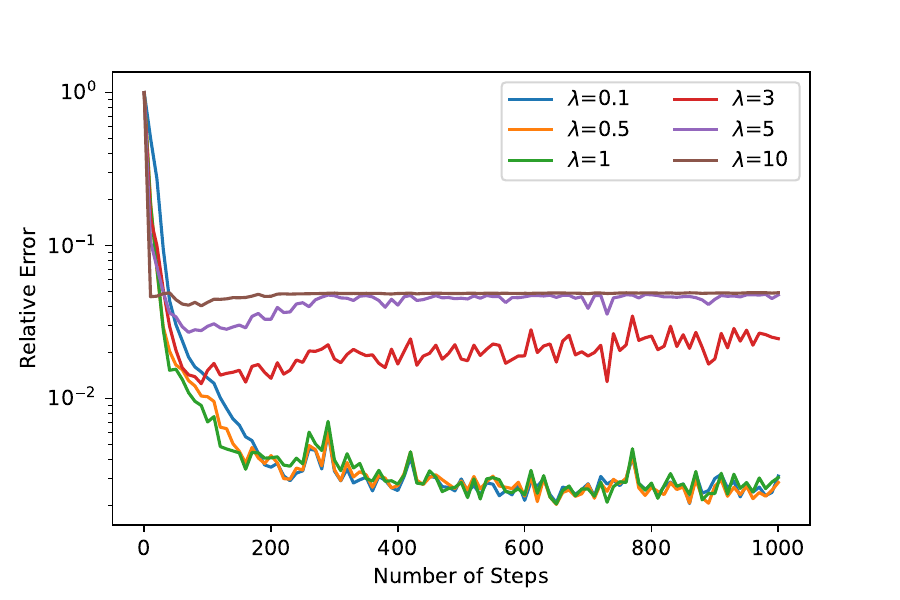}
	}
	\subfigure[$\mathcal{P}_{\mathrm{stopping}}$ by TD(0)]
	{\includegraphics[scale=0.38]{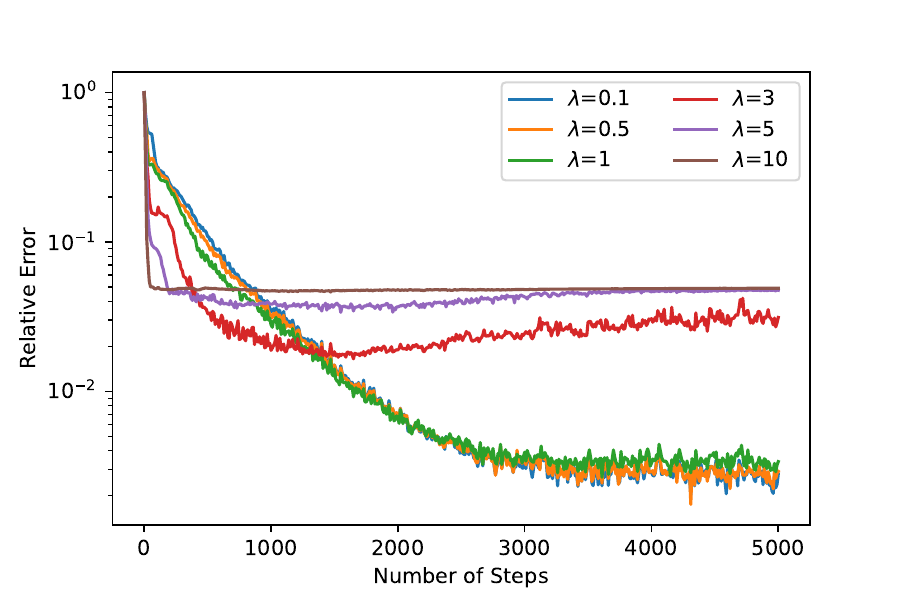}
	}
	\subfigure[$\mathcal{P}_{\mathrm{control}}$ by ML]
	{\includegraphics[scale=0.38]{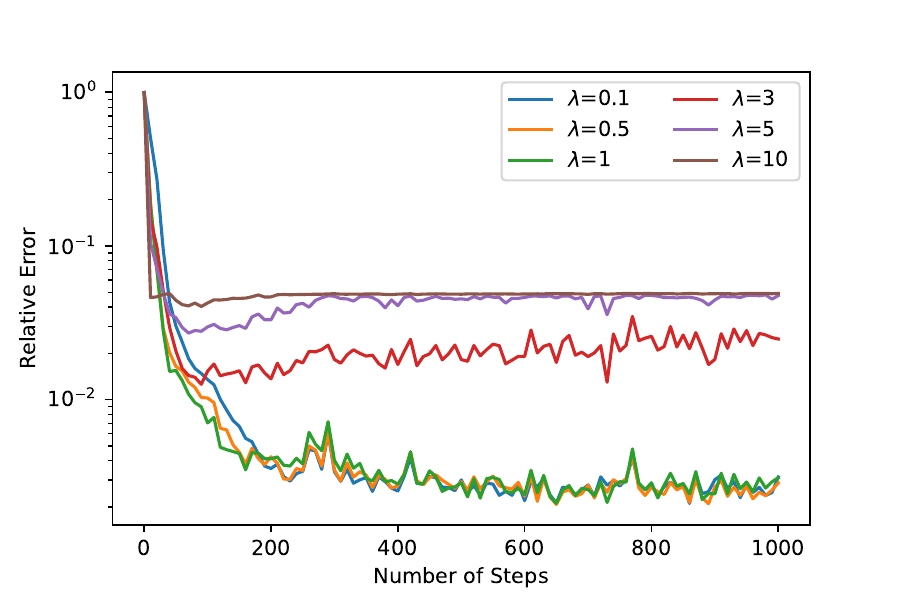}
	}
	\subfigure[$\mathcal{P}_{\mathrm{control}}$ by TD(0)]
	{\includegraphics[scale=0.38]{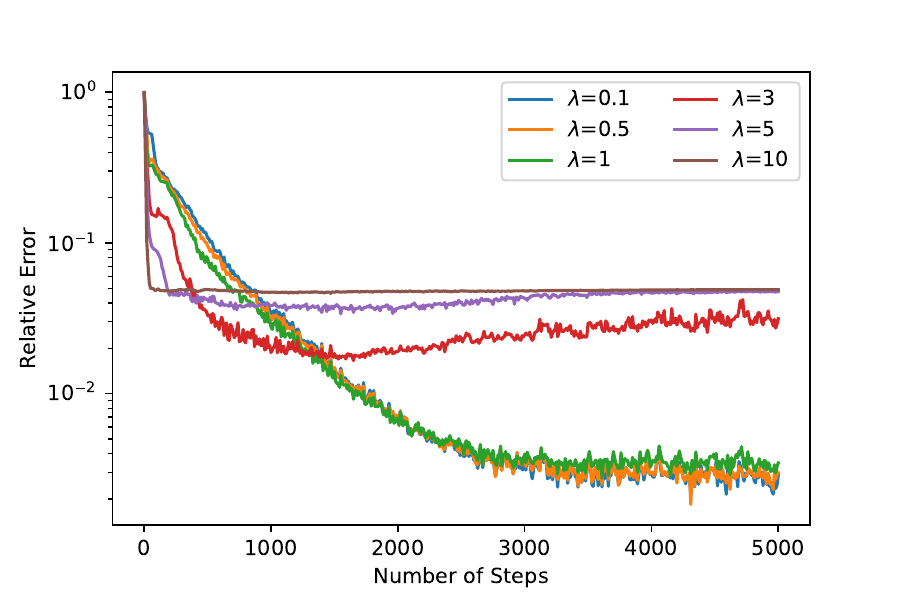}
	}
	\caption{Learning curves of option prices $\mathcal{P}_{\mathrm{stopping}}$ defined in \eqref{option_price_by_stopping} and $\mathcal{P}_{\mathrm{control}}$ defined in \eqref{option_price_by_control} by offline ML and online TD(0) algorithms. In each panel, the penalty factor $K$ is fixed to 50 and the temperature parameter $\lambda$ takes value from the set $\{0.1, 0.5, 1, 3, 5, 10\}$ while the other parameter values are set as in Table \ref{table: benchmark values}. We train the ML algorithm for 1000 steps and TD(0) algorithm for 5000 steps.} 
\label{fig: learn_curve_ML_TD0_by_VE_as_lam_varies_fix_K_49}
\end{figure}

\section{RL for Dynkin games} \label{RL_Dynkin_penalty}
Theorem 1 in \cite{dai2024learningcost} reveals that Merton's problem with proportional transaction costs is equivalent to a Dynkin game that is essentially an optimal stopping problem. Here, we focus on the Black--Scholes market where the stock price follows
\begin{equation*}
\dS_t=\alpha S_t\dt+\sigma S_t\dw_t,
\end{equation*}
where $\alpha$ and $\sigma$ are constants but unknown to the agent.
The value function of the Dynkin game satisfies the following variational inequality problem
\begin{equation}\label{HJB_Dynkin_game}
\max\left\{\min\left\{-w_t-\mathcal{L}w, w-L(x)\right\}, w-U(x)\right\}=0
\end{equation}
with terminal condition
\begin{equation}
w(T,x)=\frac{x}{x+1-\theta},
\end{equation}
where $\theta\in (0,1)$ is the proportional transaction cost rate, $x$ is the bond--stock position ratio, $\mathcal{L}$ is the generator of this ratio process, and $L(x)$ and $U(x)$ are respectively the lower and upper obstacles of $w(t,x)$.

It is proved in \cite{zhang2022solution} that \eqref{HJB_Dynkin_game} can be approximated by the following penalized PDE:
\begin{equation}
-w_t - \mathcal{L}w - K \left(w-L(x)\right)^{-}+K \left(w-U(x)\right)^{+}=0,
\end{equation}
where $K>0$ is a large penalty factor. This PDE can be equivalently written as
\begin{equation}\label{penalty_PDE_Dynkin}
w_t+\max_{a\in\left\{0, 1\right\}}\min_{b\in\left\{0, 1\right\}}\left\{\mathcal{L}w-wf(t,x,a,b)+G(t,x,a,b)\right\}=0,
\end{equation}
where
\begin{equation*}
f(t,x,a,b)=K (a+b),\quad G(t,x,a,b)=K \left(L(x)a+U(x)b\right).
\end{equation*}

Similar to the setup in Section \ref{Sec:transform_to_control}, we consider a constrained differential game:
\begin{equation}\label{prob_max_min}
w(s,x)=\max_{a\in\setu}\min_{b\in\setu} J(s,x; a,b),
\end{equation}
where
\begin{equation*}
\begin{aligned}
	J(s,x;a,b)=\mathbb{E}\Bigg[\int_s^{T} &e^{-\int_{s}^{t}f(r, X_{r}, a_{r}, b_r)\dr}G(t, X_t, a_t, b_t)\dt\nn\\ &+e^{-\int_{s}^{T}f(r, X_{r}, a_{r}, b_r)\dr}H(X_T)\;\bigg|\; X_s=x\Bigg],
\end{aligned}
\end{equation*}
and the controlled $n$-dimensional process $X$ follows
\begin{equation}
\dd X_t=b(t, X_t, a_t, b_t)\dt+\sigma(t,X_t, a_t, b_t)\dw_t.
\end{equation}
We have the following result.

\begin{theorem}[Verification theorem]
Suppose $w$ is a classical solution of
\begin{equation}
	\begin{cases}\label{HJB_max_min}
		w_{t}+\max_{a\in U}\min_{b\in U}\left\{\opla_{a,b} w-wf+G\right\}=0, ~~ (t,x)\in[0,T)\times\R^{n},\\
		w(T,x)=H(x),
	\end{cases}
\end{equation}
where
\begin{equation*}
	\opla_{a,b} w=\frac{1}{2}\sum_{i,j,k}\sigma_{i,k}(t,x,a,b)\sigma_{j,k}(t,x,a,b)\frac{\partial^2w }{\partial x_{i}\partial x_{j}}
	+\sum_{i} b_{i}(t,x,a,b)\frac{\partial w}{\partial x_{i}}.
\end{equation*}
Then it is the value function of problem \eqref{prob_max_min}.
\end{theorem}

\begin{proof}
Let $\overline{R}_{t}=e^{-\int_{s}^{t}f(r, X_{r}, a_{r}, b_{r})\dr}, \ t\geq s$.	Then $\dd \overline{R}_{t}=-\overline{R}_{t}f(t, X_{t}, a_{t}, b_{t})\dt$ and $\overline{R}_{s}=1.$	

For each fixed ${a\in U}$, let $b^*_{t}=b_t^*(x,a)$ denote the minimizer of $\opla_{a,b} w-wf+G$.
By \eqref{HJB_max_min}, $$w_{t}+\max_{a\in U}\left\{\opla_{a,b^*} w-wf(t,x,a,b^*)+G(t,x,a,b^*)\right\}=0.$$
Then, similar to the proof in Appendix \ref{proof_HJB}, we get for any admissible process $a\in\setu$,
\begin{equation*}
	\begin{aligned}
		w(s, x)
		&\geq \BE{\int_s^{T}\overline{R}_{t}^{a,b^*} G(t, X_{t}, a_t,b_t^*)\dt+\overline{R}_{T}^{a,b^*}w(T,X_{T})\;\bigg|\; X_s=x}\\
		&=\BE{\int_s^{T} e^{-\int_{s}^{t}f(r, X_{r}, a_r,b_r^*)\dr}G(t, X_t, a_t,b_t^*)\dt
			+e^{-\int_{s}^{T}f(r, X_{r}, a_r,b_r^*)\dr}H(X_T)\;\bigg|\; X_s=x}\\
		&=J(s,x;a,b^{*})\\
		& \geq \min_{b\in\setu} J(s,x;a,b).
	\end{aligned}
\end{equation*}
Since $a\in\setu$ is arbitrarily chosen, the above implies
\begin{equation*}
	w(s, x)
	\geq \max_{a\in \setu}\min_{b\in \setu} J(s,x;a,b).
\end{equation*}
Hence, $w$ is an upper bound of the value function of problem \eqref{prob_max_min}.

On the other hand, 	 if we choose a pair $(a^{*}, b^{*}_{t}(x,a^{*}))$ to solve
$$\max_{a\in U}\min_{b\in U}\left\{\opla_{a,b} w-wf+G\right\},$$ all the inequalities in the above become equalities, yielding that $w$ is the value function of \eqref{prob_max_min}. 
\end{proof}

Following the same procedure as in Section \ref{Sec:transform_to_control}, we find that the solution to \eqref{penalty_PDE_Dynkin} satisfying the terminal condition is the value function of the following stochastic differential game \citep{fleming2006controlled}:
\begin{equation}
w(s,x)=\max_{a_{t}\in\{0,1\}}\min_{b_{t}\in\{0,1\}} J(s,x;a, b),
\end{equation}
where
\begin{equation}\label{value_func_control_Dynkin}
\begin{aligned}
	J(s,x;a,b)=\mathbb{E}\Big[\int_s^{T} &Ke^{-K\int_{s}^{t}(a_r+b_r)\dr}\left(L(X_t)a_t+U(X_t)b_t\right)\dt\\\nonumber
	&+e^{-K\int_{s}^{T}(a_r+b_r)\dr}\frac{X_T}{X_T+1-\theta}\;\bigg|\; X_s=x\Big].
\end{aligned}
\end{equation}

Next, we randomize $\{a_t\}_{t\in[0,T]}$ and $\{b_t\}_{t\in[0,T]}$ respectively by $\{\nu_t\}_{t\in[0,T]}$ and $\{\mu_t\}_{t\in[0,T]}$, where
\begin{equation*}
\nu_t=\mathbb{P}\left(a_t=1\right)=1 - \mathbb{P}\left(a_t=0\right), \quad \mu_t=\mathbb{P}\left(b_t=1\right)=1 - \mathbb{P}\left(b_t=0\right).
\end{equation*}
It turns out that the exploratory value function of the Dynkin game is
\begin{equation}\label{value_Dynkin_randomized}
w^{\lambda}(s,x)=\max_{\nu_{t}\in[0,1]}\min_{\mu_{t}\in[0,1]} J^{\lambda}(s,x;\nu, \mu),
\end{equation}
where
\begin{equation}\label{value_func_control_Dynkin_randomize}
\begin{aligned}
	J^{\lambda}(s,x;\nu,\mu)=\mathbb{E}\Big[\int_s^{T} &Ke^{-K\int_{s}^{t}(\nu_r+\mu_r)\dr}\left[L(X_t)\nu_t+U(X_t)\mu_t-\lam \left(\BH(\nu_{t}) - \BH(\mu_t)\right)\right]\dt\\\nonumber
	&+e^{-K\int_{s}^{T}(\nu_r+\mu_r)\dr}\frac{X_T}{X_T+1-\theta}\;\bigg|\; X_s=x\Big].
\end{aligned}
\end{equation}
The HJB equation of problem \eqref{value_Dynkin_randomized} is thus
\begin{equation}
w^{\lambda}_t+\mathcal{L}w^{\lambda}+\max_{\nu\in[0,1]}\min_{\mu\in[0,1]}\left\{-w^{\lambda}K(\nu+\mu)+K\left(L(x)\nu+U(x)\mu\right)-\lambda\left(\BH(\nu)-\BH(\mu)\right)\right\}=0,
\end{equation}
and the optimal solutions are
\begin{equation}
\begin{aligned}\label{optimal_prob_nu_mu_Dynkin}
	\nu^* &=\frac{\exp\left(\frac{K}{\lambda}(L(x)-w^{\lambda})\right)}{\exp\left(\frac{K}{\lambda}(L(x)-w^{\lambda})\right)+1}=\frac{1}{1+\exp\left(-\frac{K}{\lambda}(L(x)-w^{\lambda})\right)},\\
	\mu^* &=\frac{\exp\left(\frac{K}{\lambda}(w^{\lambda}-U(x))\right)}{\exp\left(\frac{K}{\lambda}(w^{\lambda}-U(x))\right)+1}=\frac{1}{1+\exp\left(-\frac{K}{\lambda}(w^{\lambda}-U(x))\right)}.
\end{aligned}
\end{equation}

Consequently, if we choose the TD error criterion for policy evaluation, the RL algorithm is essentially the same as Algorithm 1 in \cite{dai2024learningcost}, except that $\nu, \mu, \nu^*, \mu^*$ in their algorithm should be replaced by $K\nu, K\mu$, $\nu^*, \mu^*$ in \eqref{optimal_prob_nu_mu_Dynkin} in this section, and their entropy term should be replaced by the differential entropy in this paper.

\section{Finite difference method}\label{finite_difference}

In this section, we present (classical, model-based) algorithms for pricing a finite-horizon American put option and obtaining its stopping boundary. We use the penalty method in \cite{forsyth2002quadratic}. The payoff function is $g(x)=\left(\Gamma-x\right)^{+}$ with $x$ being the stock price and $\Gamma$ the strike price.
We consider the PDE \eqref{hjb2} and let $u(t,y) := v(t,x)$ with $x=e^{y}$.
In the finite difference scheme below, the time horizon $\left[0,T\right]$ is divided to $N_t$ equally spaced intervals with $t_i=i\Delta t,\ \text{for } i=0,\ldots, N_t$, and the range for $y$ is chosen to be $\left[-N, N\right]$ with $N$ being a large positive number, which is further divided into $N_y$ equally spaced intervals with $y_j = -N + j\Delta y\ \text{for } j = 0,\ldots,N_y$.
Thus, $u_{i,j}$ denotes the numerical value at the mesh point $\left(t_i, y_j\right)$, namely, $u_{i,j}:=u\left(t_i,y_j\right)$ for $i=0,\ldots,N_t, j=0,\ldots,N_y$.

In the following, let $u_{i,j}^k$ denotes the $k$-th iteration value of $u_{i,j}$. The terminal condition is
\begin{equation}\label{u_terminal}
u_{N_t,j} = g(e^{y_j})
\end{equation}
for $j=0,\ldots,N_y$.
Similarly, the boundary conditions at $y_0$ and $y_{N_y}$ are set to be
\begin{equation}\label{u_boundary}
u_{i,0} = g(e^{y_0}),\ \text{and }	u_{i,N_y} = g(e^{y_{N_y}})
\end{equation}
for $i=0,\ldots,N_t$.
The Black--Scholes operator $v_t(t,x)+\frac{1}{2}\sigma^2x^2v_{xx}(t,x)+\rho xv_x(t,x)-\rho v(t,x)$ changes to
\begin{equation*}
\mathcal{L}_{BS}u(t,y):=u_{t}(t,y)+\frac{1}{2}\sigma^2u_{yy}(t,y)+\left(\rho-\frac{1}{2}\sigma^2\right)u_y(t,y)-\rho u(t,y).
\end{equation*}

By the fully implicit finite difference scheme,
\begin{equation*}
\begin{aligned}
	\frac{\partial}{\partial t} u(t_i, y_j) &= \frac{u_{i+1,j} - u_{i,j}}{\Delta t},\\
	\frac{\partial}{\partial y} u(t_i, y_j) &= \frac{u_{i, j+1}-u_{i, j-1}}{2\Delta y},\\
	\frac{\partial^2}{\partial y^2} u(t_i, y_j) &= \frac{u_{i, j+1} - 2u_{i,j} + u_{i, j-1}}{\left(\Delta y\right)^2},
\end{aligned}
\end{equation*}
for $j=1,\ldots,N_y-1$ and $i=0,\ldots,N_t-1$. The discretized Black--Scholes operator is therefore
\begin{equation}\label{L_BS_finite}
\mathcal{L}_{BS} u_{i,j} = \frac{u_{i+1,j}-u_{i,j}}{\Delta t} + \beta u_{i, j+1} - \left(\alpha+\beta+\rho\right) u_{i,j} + \alpha u_{i, j-1},
\end{equation}
where
\begin{equation*}
\alpha = \frac{\sigma^2}{2\left(\Delta y\right)^2} - \frac{\rho - \frac{1}{2}\sigma^2}{2\Delta y},\ \text{and } \beta = \frac{\sigma^2}{2\left(\Delta y\right)^2} + \frac{\rho - \frac{1}{2}\sigma^2}{2\Delta y}.
\end{equation*}

The PDE in \eqref{hjb2} becomes
\begin{equation}\label{ti_yj_PDE}
-\mathcal{L}_{BS} u(t,y) = K \left(g(e^y)-u(t,y)\right)^{+}.
\end{equation}
For each mesh point $\left(t_i, y_j\right)$, we use the following approximation by Newton iteration:
\begin{equation}\label{NewtonIter_VI}
\begin{aligned}
	\left(g(e^{y_j})-u_{i,j}^{k+1}\right)^{+} & = \left(g(e^{y_j})-u_{i,j}^{k}+u_{i,j}^{k}-u_{i,j}^{k+1}\right)^{+}\nonumber\\
	& \approx \left(g(e^{y_j})-u_{i,j}^{k}\right)^{+} + 1_{g(e^{y_j})-u_{i,j}^{k} > 0} \left(u_{i,j}^{k}-u_{i,j}^{k+1}\right)\nonumber\\
	& = \left(g(e^{y_j})-u_{i,j}^{k+1}\right)1_{g(e^{y_j})-u_{i,j}^{k} > 0},
\end{aligned}
\end{equation}
where $1_{A}$ is the indicator function of a set $A$.
By \eqref{L_BS_finite} and \eqref{NewtonIter_VI}, it turns out that \eqref{ti_yj_PDE} reads
\begin{equation*}
- \frac{u_{i+1,j}-u_{i,j}^{k+1}}{\Delta t} - \beta u_{i, j+1}^{k+1} + \left(\alpha+\beta+\rho\right) u_{i,j}^{k+1} - \alpha u_{i, j-1}^{k+1} =
K \left(g(e^{y_j})-u_{i,j}^{k+1}\right)1_{g(e^{y_j})-u_{i,j}^{k} > 0},
\end{equation*}
which simplifies to
\begin{equation}\label{NewtonIter_equ}
- \beta u_{i, j+1}^{k+1} + \left(\frac{1}{\Delta t} + \alpha+\beta+ \rho + K 1_{g(e^{y_j})-u_{i,j}^{k} > 0}\right) u_{i,j}^{k+1} - \alpha u_{i, j-1}^{k+1} = \frac{u_{i+1,j}}{\Delta t} +
K g(e^{y_j}) 1_{g(e^{y_j})-u_{i,j}^{k} > 0}.
\end{equation}
In each iterative step, the quantities in the RHS of \eqref{NewtonIter_equ} regarding future values $u_{i+1,j}$ for all $j$ and previous estimates $u_{i,j}^k$ for all $i$ and $j$ are known, and the current estimates $u_{i,j}^{k+1}$ are obtained by solving the equation \eqref{NewtonIter_equ}. See Algorithm \ref{algorithm_am_put} for details.
\begin{algorithm}\label{algorithm_am_put}
\caption{Finite difference scheme for \eqref{hjb2}}
\textbf{Inputs:} $\Delta t, \Delta y, N, N_t, N_y, K$ and the tolerance level $\ep>0$.\\
\textbf{Learning procedure:}\\
Initialize $u_{N_t,j}$ for all $j$ by \eqref{u_terminal} as well as $u_{i,0}$ and $u_{i,N_y}$ for all $i$ by \eqref{u_boundary}.\\
\For {$i = N_t-1,\ldots,0$}{
	Set $u_{i, j}^0 = u_{i+1, j}$ for $j=1,\ldots,N_y-1$.\\
	\For{$k=0,1,\dots$}{
		Solve the equation \eqref{NewtonIter_equ} to get $u_{i,j}^{k+1}$ for $j=1,\ldots,N_y-1$.\\
		\If{$ {||u_i^{k+1}-u_i^k||_{\infty}}< \ep* {\max\left\{1,\ ||u_i^k||_{\infty}\right\}} $}
		{Quit the loop for $k$.}			
	}
	Set $u_{i,j}=u_{i,j}^{k+1}$ for $j=1,\ldots,N_y-1$.\\
}
\end{algorithm}

\phantomsection
\addcontentsline{toc}{section}{References}
\bibliographystyle{apalike}

\end{document}